\documentclass{article}
\usepackage[utf8]{inputenc}
\usepackage{amsmath}
\usepackage{amssymb}
\usepackage{amsthm}
\usepackage{graphicx}
\usepackage[margin=1.4cm]{geometry}

\newtheorem{thm}{Theorem}

\newtheorem{conj}{Conjecture}
\newtheorem{prop}{Proposition}
\newtheorem{lem}{Lemma}
\newtheorem{rmk}{Remark}

\title{$SO(2)\times SO(3)$-invariant Ricci solitons and ancient flows on $\mathbb{S}^4$}
\author{Timothy Buttsworth}
\date{}

\begin{document}

\maketitle

\abstract{Consider the standard action of $SO(2)\times SO(3)$ on $\mathbb{R}^5=\mathbb{R}^2\oplus \mathbb{R}^3$. 
We establish the existence of a uniform constant $\mathcal{C}>0$ so that any $SO(2)\times SO(3)$-invariant Ricci soliton 
on $\mathbb{S}^4\subset \mathbb{R}^5$ with Einstein constant $1$ must have Riemann curvature and volume bounded by 
$\mathcal{C}$, and injectivity radius bounded below by $\frac{1}{\mathcal{C}}$. 
This observation, coupled with basic numerics, gives 
strong evidence to suggest that the only $SO(2)\times SO(3)$-invariant Ricci solitons on $\mathbb{S}^4$ are round.
We also encounter the so-called `pancake' ancient solution of the Ricci flow.}

\section{Introduction}
Let $M$ be a smooth manifold. A \textit{gradient Ricci soliton} on $M$ is a Riemannian metric $g$ for which there exists a constant $\lambda$ and a smooth function $u:M\to 
\mathbb{R}$ so that
\begin{equation}\label{GRS}
 \text{Ric} (g)+Hess_g(u)=\lambda g.
\end{equation}
Solutions of \eqref{GRS} arise as self-similar solutions to the well-known Ricci flow 
\begin{align}\label{RF}
 \frac{\partial g}{\partial t}=-2\text{Ric}(g). 
\end{align}
In the quest for new solutions to \eqref{GRS}, one often assumes that $g$ and $u$ are invariant under a certain group action $G$ of $M$. 
It is of no use to assume that $G$ acts transitively on $M$, because then $g$ is homogeneous and $u$ is constant, so $g$ must be Einstein. 
Finding homogeneous Einstein metrics is its own well-studied problem; perhaps the crowning achievement in this area of study is the work in \cite{BWZ}, where the authors 
show that the mountain pass theorem is quite generally applicable 
in the construction of compact homogeneous Einstein metrics. The classification of non-compact homogeneous Einstein metrics is the subject of the long-standing 
Alekseevskii conjecture (see Conjecture 7.57 of \cite{Besse}). 

After assuming that $G$ acts transitively, the next natural step is 
assuming that $G$ acts with \textit{cohomogeneity one}, which means that the generic orbits of the action of $G$ in $M$ have dimension one less than that of the manifold. 
Several examples of gradient Ricci solitons have been constructed using cohomogeneity one invariance (perhaps the most notable examples are \cite{Bohm98} and \cite{DancerWang}).
In this paper, 
we consider the problem of solving \eqref{GRS} on $\mathbb{S}^4$ with $\lambda=1$ 
for a pair $(g,u)$ which is invariant under the usual cohomogeneity one 
action of $SO(2)\times SO(3)$.
We first show a compactness result for solutions of this problem. 
\begin{thm}\label{CT}
 There exists a $\mathcal{C}>0$ so that any $SO(2)\times SO(3)$-invariant solution $(g,u)$ of \eqref{GRS} (with $\lambda=1$) on $\mathbb{S}^4$
 has $\text{inj}_g\ge \frac{1}{\mathcal{C}}$, $\text{vol}(g)\le \mathcal{C}$, and Riemann curvature bounded pointwise by $\mathcal{C}$. 
\end{thm}
Theorem \ref{CT} 
forms part of a well-established area of study which seeks to produce various types of compactness results for spaces of gradient shrinking 
Ricci solitons, especially on four-dimensional manifolds. The strongest general result to date appears 
to be Theorem 1.1 in \cite{HaslMuller}, which shows compactness in the orbifold sense (with the pointed Cheeger-Gromov topology), 
but only once a uniform lower bound on the Perelman entropy is known. The proof of this orbifold compactness result
essentially boils down to establishing a uniform $L^2$ norm on the Riemann curvature, 
rather than the stronger uniform $L^{\infty}$ bounds we establish with Theorem \ref{CT}. 

We hope that Theorem \ref{CT} brings us a step closer to actually determining \textit{uniqueness} of invariant Ricci solitons. 
\begin{conj}\label{ClT}
 Any $SO(2)\times SO(3)$-invariant solution of \eqref{GRS} on $\mathbb{S}^4$ with $\lambda=1$ is the round sphere, up to diffeomorphism. 
\end{conj}
The prospects of verifying this conjecture numerically are discussed in Section \ref{mainproof}. 

In the course of proving Theorem \ref{CT}, we discover that our notion of compactness is not very rigid, in the sense 
that we find a sequence of `almost' Ricci solitons with unbounded Riemann curvature. These `almost' solitons are an interpolation between 
the Gaussian shrinker on $\mathbb{R}^2\times \mathbb{S}^2$ and a rescaled product of the Bryant soliton on $\mathbb{R}^3$ with a flat metric on $\mathbb{S}^1$. The only 
reason we can conclude that these are \text{not} 
Ricci solitons is that these metrics have non-negative Riemann curvature; if these metrics were solitons, they would be round by 
Hamilton's pinching results in \cite{Hamilton86}. However, it 
turns out that this `pancake' shape is a 
$\kappa$-noncollapsed ancient Ricci flow on $\mathbb{S}^4$ with positive Riemann curvature. 
\begin{thm}\label{NAS}
 There exists a $\kappa>0$ and a $\kappa$-noncollapsed, non-round $SO(2)\times SO(3)$-invariant ancient solution to the Ricci flow  on $\mathbb{S}^4$ 
 with positive Riemann curvature operator. 
\end{thm}
\section*{Acknowledgements}
I would like to thank David Buttsworth for guidance in implementing some of the numerical aspects of this project. 
I would also like to thank Max Hallgren, Mat Langford, Jason Lotay 
and Yongjia Zhang for several useful conversations regarding the `pancake' ancient solution of the Ricci flow. 
\section{Preliminaries}
In case the metric $g$ and the function $u$ are invariant under a certain cohomogeneity one action, the Ricci soliton equation \eqref{GRS} reduces 
to a system of ordinary differential equations. In this section, we discuss these equations in the case that our cohomogeneity one action 
is $SO(2)\times SO(3)$, and we provide some initial results on their solutions using the maximum principle. 
It turns out that all of the material in this section applies to solitons on both $\mathbb{S}^4$ and $\mathbb{S}^2\times \mathbb{S}^2$ which 
are invariant under the obvious action of $SO(2)\times SO(3)$, so for 
this section only, we consider both of these four-dimensional manifolds. 
\subsection{The boundary value problem}
Under the action of $SO(2)\times SO(3)$, the principal orbits of $\mathbb{S}^4$ or $\mathbb{S}^2\times \mathbb{S}^2$ 
are product spheres $\mathbb{S}^1\times \mathbb{S}^2$. In the case of $\mathbb{S}^4$, the two singular orbits 
are one copy of $\mathbb{S}^1$ and one copy of $\mathbb{S}^2$, whereas for $\mathbb{S}^2\times \mathbb{S}^2$, both singular orbits are copies of $\mathbb{S}^2$. 
Up to diffeomorphism, any $SO(2)\times SO(3)$-invariant Riemannian metric on $\mathbb{S}^4$ or $\mathbb{S}^2\times \mathbb{S}^2$ has the form 
\begin{equation}\label{metricform}
 dt^2+f_1^2 d\theta^2+f_2^2Q,
\end{equation}
 where $f_i:(0,T)\to \mathbb{R}$ are smooth and positive functions, $d\theta$ is the standard one-form on $\mathbb{S}^1$, and 
 $Q$ is the round metric of unit Ricci curvature on $\mathbb{S}^2$. In order for the Riemannian metric $g$ to close up smoothly at 
 the singular orbits, the functions $f_1$ and $f_2$ must be smoothly extendable to functions on $(-T,2T)$ so that
\begin{align}\label{metricsmoothness}
\begin{split}
 f_1(0)>0\ \text{and $f_1(t)$ is even about $t=0$},\qquad
 f_2'(0)=1\ \text{and $f_2(t)$ is odd about $t=0$},\\
f_1'(T)=-1\ \text{and $f_1(t)$ is odd about $t=T$},\qquad
f_2(T)>0\ \text{and $f_2(t)$ is even about $t=T$},\\
 \end{split}
\end{align}
in the case of $\mathbb{S}^4$, and so that 
\begin{align}\label{metricsmoothness'}
\begin{split}
 f_1'(0)=1\ \text{and $f_1(t)$ is odd about $t=0$},\qquad
 f_2(0)>0\ \text{and $f_2(t)$ is even about $t=0$},\\
f_1'(T)=-1\ \text{and $f_1(t)$ is odd about $t=T$},\qquad 
f_2(T)>0\ \text{and $f_2(t)$ is even about $t=T$},\\
 \end{split}
\end{align}
in the case of $\mathbb{S}^2\times \mathbb{S}^2$. 
At any point in a principal orbit, let $\{e_1,e_2,e_3\}$ be an orthonormal basis adapted to the standard product metric $d\theta^2+Q$ on $\mathbb{S}^1\times \mathbb{S}^2$. 
Then in the basis 
$\{\partial_t\wedge e_1,\partial_t\wedge e_2,\partial_t\wedge e_3,e_2\wedge e_3,e_1\wedge e_2, e_1\wedge e_3\}$, we compute (cf. \cite{GroveZiller}) 
the Riemann curvature operator:
\begin{align}\label{riemanncurvatureoperator}
 \mathcal{R}=\begin{pmatrix}
              -\frac{f_1''}{f_1}&0&0&0&0&0\\
              0&-\frac{f_2''}{f_2}&0&0&0&0\\
              0&0&-\frac{f_2''}{f_2}&0&0&0\\
              0&0&0&\frac{1-f_2'^2}{f_2^2}&0&0\\
              0&0&0&0&-\frac{f_1'f_2'}{f_1f_2}&0\\
              0&0&0&0&0&-\frac{f_1'f_2'}{f_1f_2}
             \end{pmatrix}.
\end{align}
Suppose we have an $SO(2)\times SO(3)$-invariant metric of the form \eqref{metricform} which is a gradient shrinking Ricci soliton with an 
$SO(2)\times SO(3)$-invariant potential function $u$. 
Then $u$ depends only on the $t$ parameter, and can be smoothly extended to a function on $(-T,2T)$ so that 
\begin{align}\label{usmoothness}
 u(t) \ \text{is even about both $t=0$ and $t=T$},
\end{align}
and the functions $(f_1,f_2,u)$ satisfy \eqref{metricsmoothness} or \eqref{metricsmoothness'}, alongside 
the Ricci soliton equation \eqref{GRS} on $(0,T)$ which becomes
\begin{align}\label{solitonequations}
\begin{split}
 -\frac{f_1''}{f_1}-2\frac{f_2''}{f_2}+u''&=\lambda,\\
 -\frac{f_1''}{f_1}-2\frac{f_1'f_2'}{f_1f_2}+u'\frac{f_1'}{f_1}&=\lambda,\\
 -\frac{f_2''}{f_2}-\frac{f_1'f_2'}{f_1f_2}+\frac{1-f_2'^2}{f_2^2}+u'\frac{f_2'}{f_2}&=\lambda. 
 \end{split}
\end{align}
We find it useful to set $\lambda=1$ and introduce the new functions $L_i=\frac{f_i'}{f_i}$, $\xi=L_1+2L_2-u'$, $R=\frac{1}{f_2}$, so the 
Ricci soliton equations in these variables are 
\begin{align}\label{equationsnewf}
\begin{split}
 \xi'=-L_1^2-2L_2^2-1,\\
 L_1'=-\xi L_1-1,\\
 L_2'=-\xi L_2+R^2-1,  \\
 R'=-L_2 R.
 \end{split}
\end{align}
Using the boundary conditions, a solution of \eqref{equationsnewf} on $(0,T)$ uniquely determines a solution of \eqref{solitonequations}. 
Note that in these co-ordinates, the Riemann curvature eigenvalues for a Ricci soliton are:
\begin{itemize}
 \item $\xi L_1+1-L_1^2$ (we know this must be non-negative for a soliton by Proposition \ref{scf1} below);
 \item $R^2-L_2^2$ (we know this must be non-negative for a soliton by Proposition \ref{sf2p} below);
 \item $-L_1L_2$ (occurs twice); and 
 \item $\xi L_2+1-R^2-L_2^2$ (occurs twice). 
\end{itemize}

\subsection{An initial step towards compactness: the maximum principle}
If we were lucky enough to know \textit{a priori} that any $SO(2)\times SO(3)$-invariant shrinking soliton on $\mathbb{S}^4$ had to 
have non-negative Riemann curvature operator, then Hamilton's pinching results for $4$-manifolds under the Ricci flow 
in \cite{Hamilton86} would prove Conjecture \ref{ClT} in the affirmative 
immediately (and would therefore give us the proof of Theorem \ref{CT} as well). Although there does not seem to be a way to cheaply show that our Ricci solitons must 
have non-negative Riemann curvature, 
a relatively simple application of the maximum principle does show that at least two of the six Riemann curvature eigenvalues must be non-negative everywhere. 
Since this observation is straightforward, and is frequently used in the remainder of the paper, we include the proofs in this preliminary section. 
\begin{prop}\label{scf1}
 For any $SO(2)\times SO(3)$-invariant gradient shrinking Ricci soliton $(g,u)$ on $\mathbb{S}^4$ or $\mathbb{S}^2\times \mathbb{S}^2$ 
 of the form \eqref{metricform}, 
 the quantity $\frac{-f_1''}{f_1}$ is non-negative everywhere. 
\end{prop}
\begin{proof}
For a given point on a principal orbit, consider the selfdual/anti-selfdual basis 
\begin{align*}
 \{\partial_t\wedge e_1+e_2\wedge e_3,\partial_t\wedge e_2+e_1\wedge e_3,e_3\wedge \partial_t+e_1\wedge e_2,
 \partial_t\wedge e_1-e_2\wedge e_3,\partial_t\wedge e_2-e_1\wedge e_3,e_3\wedge \partial_t-e_1\wedge e_2\}.
\end{align*}
 In this basis, the Riemann curvature operator is given by 
 \begin{align*}
  \mathcal{R}=\begin{pmatrix}
               A & B\\
               B^T &A
              \end{pmatrix}
 \end{align*}
 where 
 \begin{align*}
  2A=\begin{pmatrix}
     -\frac{f_1''}{f_1}+\frac{1-(f_2')^2}{f_2^2}&0&0\\
     0&-\frac{f_2''}{f_2}-\frac{f_1'f_2'}{f_1f_2}&0\\
     0&0&-\frac{f_2''}{f_2}-\frac{f_1'f_2'}{f_1f_2}
    \end{pmatrix},
\ 
2B=
   \begin{pmatrix}
     -\frac{f_1''}{f_1}-\frac{1-(f_2')^2}{f_2^2}&0&0\\
     0&-\frac{f_2''}{f_2}+\frac{f_1'f_2'}{f_1f_2}&0\\
     0&0&-\frac{f_2''}{f_2}+\frac{f_1'f_2'}{f_1f_2}
    \end{pmatrix}.
 \end{align*} 
 After applying the Uhlenbeck trick, Hamilton shows in \cite{Hamilton86} that, under the Ricci flow, the curvatures satisfy the evolution equations 
 \begin{align*}
  \frac{\partial A}{\partial \tau}=\Delta_{g(\tau)} A+A^2+B^T B+2A^\#, \qquad \frac{\partial B}{\partial \tau}=\Delta_{g(\tau)} B+AB+BA+2B^\#.
 \end{align*}
 Since the second and third diagonal entries of both $A$ and $B$ are identical, the first entries of $A^\#$ and $B^\#$ are non-negative. 
 We therefore find that the first eigenvalue of the matrix $A+B$ must be increasing under the Ricci flow. 
 However, the Riemann curvatures of gradient shrinking Ricci solitons evolve only by diffeomorphisms and scalings under the flow, so 
 the first eigenvalue of the matrix $A+B$, which is $-\frac{f_1''}{f_1}$, 
must be non-negative everywhere. 
\end{proof}
\begin{prop}\label{sf2p}
 For any $SO(2)\times SO(3)$-invariant gradient shrinking Ricci soliton 
 on $\mathbb{S}^4$ or $\mathbb{S}^2\times \mathbb{S}^2$ of the form \eqref{metricform}, 
 the quantity $\frac{1-f_2'^2}{f_2^2}$ is non-negative everywhere. More generally, suppose we have an 
 $SO(2)\times SO(3)$-invariant Riemannian metric of the form \eqref{metricform}, as well as two points $r_1,r_2\in [0,T]$ so that:
 \begin{itemize}
  \item $\frac{1-f_2'^2}{f_2^2}\ge 0$ at $r_1,r_2$; and 
  \item the metric satisfies the Ricci soliton equations on $(r_1,r_2)$.
 \end{itemize}
Then $\frac{1-f_2'^2}{f_2^2}\ge 0$ on $(r_1,r_2)$. 
\end{prop}
\begin{proof}
 Clearly it suffices to show that $f_2'\in [-1,1]$ on $[r_1,r_2]$. 
 Differentiating and using the equations of \eqref{solitonequations}, we find 
 \begin{align*}
  0&=-f_2'''+(-\frac{f_1'f_2'}{f_1}+\frac{1-f_2'^2}{f_2})'+(u''-\lambda)f_2'+u'f_2''\\
  &=-f_2'''+\frac{f_1'^2f_2'}{f_1^2}-\frac{f_1'f_2''}{f_1}+\frac{-2f_2'f_2''f_2+f_2'(f_2'^2-1)}{f_2^2}+2\frac{f_2''f_2'}{f_2}+u'f_2''.\\
 \end{align*}
Therefore, if $f_2'>1$ somewhere on $(r_1,r_2)$, then there must be a point in $(r_1,r_2)$ with $f_2'>1$, $f_2''=0$ and $f_2'''\le 0$. At this point, we find 
\begin{align*}
 0&=-f_2'''+\frac{f_1'^2f_2'}{f_1^2}+\frac{f_2'(f_2'^2-1)}{f_2^2}>0,
\end{align*}
a contradiction. We obtain a similar contradiction if there were a point with $f_2'<-1$. A metric which is a Ricci soliton everywhere has $f_2'\in [-1,1]$ at 
$0$ and $T$ for both $\mathbb{S}^4$ and $\mathbb{S}^2\times \mathbb{S}^2$ by \eqref{metricsmoothness} and \eqref{metricsmoothness'}, so the claim follows. 
\end{proof}
\section{The shooting problem for $\mathbb{S}^4$: a curve and a surface}
We turn to the problem of establishing Theorem \ref{CT}. 
It is convenient to cast the problem of finding solutions of \eqref{metricsmoothness}, \eqref{usmoothness} and  
 \eqref{equationsnewf} as a shooting problem: the idea is to 
study the initial value problem for \eqref{equationsnewf} around the two singular orbits at $t=0,T$, and examine how the solutions meet at a specified principal orbit. 
In particular, 
we will examine how the solutions meet at an orbit where $\xi=0$, since \eqref{equationsnewf} implies that there will be exactly one such orbit for a shrinking soliton. 

To begin, we examine more carefully the initial value problem at the $\mathbb{S}^1$ orbit ($t=0$). Note that, by \eqref{metricsmoothness} and \eqref{usmoothness}, 
there must be smooth functions $\eta_0,\eta_1,\eta_2,\eta_3$ so that 
\begin{align}\label{etaforma}
 \xi(t)=\frac{2}{t}+\eta_0(t), \qquad L_1(t)=\eta_1(t), \qquad L_2(t)=\frac{1}{t}+\eta_2(t), \qquad R(t)=\frac{1}{t}+\eta_3(t),
\end{align}
and $\eta_i(0)=0$ for all $i=0,1,2,3$. It turns out that whenever solving \eqref{equationsnewf} subject to \eqref{etaforma} close to $t=0$, solutions are uniquely determined by 
the number 
\begin{align}\label{ivpa}
 \delta_1=\eta_3'(0)=\lim_{t\to 0}\left(\frac{R(t)-\frac{1}{t}}{t}\right).
\end{align}
\begin{prop}\label{mapa}
 For each value of $\delta_1\in \mathbb{R}$, there exists a unique solution to the initial value problem of \eqref{equationsnewf}, subject to \eqref{etaforma}
 and \eqref{ivpa}. 
 The solution can be extended smoothly to a point with $\xi=0$, and $\xi^{-1}(0)$ depends smoothly on $\delta_1$. 
 For $i=0,1,2,3$, the quantity $\frac{\eta_i(t)}{t}$ depends continuously on $\delta_1$ and $t\in [0,\xi^{-1}(0)]$. 
\end{prop}
The proof of Proposition \ref{mapa} essentially follows from the techniques discussed in \cite{Buzano}. 
We can similarly examine the initial value problem at the $\mathbb{S}^2$ orbit. This time, 
\eqref{metricsmoothness} and \eqref{usmoothness} imply that 
there must be smooth functions $\eta_0,\eta_1,\eta_2,\eta_3$ so that 
\begin{align}\label{etaformb}
 -\xi(T-t)=\frac{1}{t}+\eta_0(t), \qquad -L_1(T-t)=\frac{1}{t}+\eta_1(t), \qquad -L_2(T-t)=\eta_2(t), \qquad R(T-t)=R(T)+\eta_3(t),
\end{align}
where again $\eta_i(0)=0$ for $i=0,1,2,3$. 
This time, solutions are uniquely determined by 
\begin{align}\label{ivpb}
 \delta_2=\eta_0'(0), \qquad \delta_3=R(T).
\end{align}
\begin{prop}\label{mapb}
 For each value of $(\delta_2,\delta_3)\in \mathbb{R}^2$, there exists a unique solution to the initial value problem of \eqref{equationsnewf}, subject to 
 \eqref{etaformb} and \eqref{ivpb}.
 The solution can be extended smoothly to a point with $\xi=0$, and $\xi^{-1}(0)$ depends smoothly on $(\delta_2,\delta_3)$.
 For $i=0,1,2,3$, the quantity $\frac{\eta_i(t)}{t}$ depends continuously on $\delta_2,\delta_3$ and $t\in [\xi^{-1}(0),T]$.
\end{prop}

The main idea behind the proof of Theorem \ref{CT} is to provide bounds on the possible values that $\delta_1,\delta_2,\delta_3$ can achieve. Fortunately, we can find some 
of these bounds immediately. 
\begin{prop}\label{initialbounds}
 If the solutions to the IVPs found in Propositions \ref{mapa} and \ref{mapb} coincide at time $\xi^{-1}(0)$, then $\delta_1\ge 0$, $\delta_2\ge -1$ and 
 $\delta_3\ge 0$. 
\end{prop}
\begin{proof}
 The $\delta_3$ bound is obvious because $\lim_{t\to 0}R(t)=\infty$, and $R'=-L_2R$. For $\delta_1$, note that $R^2-L_2^2\ge 0$ by Proposition \ref{sf2p}, so 
 using \eqref{etaforma} and \eqref{ivpa}, we find $0\le \lim_{t\to 0}(R(t)^2-L_2(t)^2)=2\eta_3'(0)-2\eta_2'(0)$. 
 But by using \eqref{equationsnewf} we see that $\eta_2'(0)=-2\eta_3'(0)$, so $0\le 6\delta_1$. 
 For $\delta_2$, we can use Proposition \ref{scf1} to find that $\xi L_1+1-L_1^2\ge 0$ everywhere; the result then follows similarly to the analogous result for 
 $\delta_1$. 
\end{proof}

Let $C$ be the smooth curve in $\mathbb{R}^3$ consisting of all values of $(L_1,L_2,R)$ evaluated at $\xi^{-1}(0)$ found from Proposition \ref{mapa} 
for $\delta_1\ge 0$, and let 
$S$ be the smooth surface in $\mathbb{R}^3$ consisting of all values of $(L_1,L_2,R)$ evaluated at $\xi^{-1}(0)$ found from Proposition \ref{mapb} for $\delta_2\ge -1$ and $\delta_3\ge 0$. 
Proposition \ref{initialbounds} implies that any $SO(2)\times SO(3)$-invariant Ricci soliton on $\mathbb{S}^4$ must correspond to a point in $\mathbb{R}^3$ 
where the curve $C$ intersects the surface $S$. The following images give various views of an approximation to $C$ (in red) and $S$ (in blue) which were found 
using Matlab's ODE solver. 

\begin{center}
    \includegraphics[width=5cm]{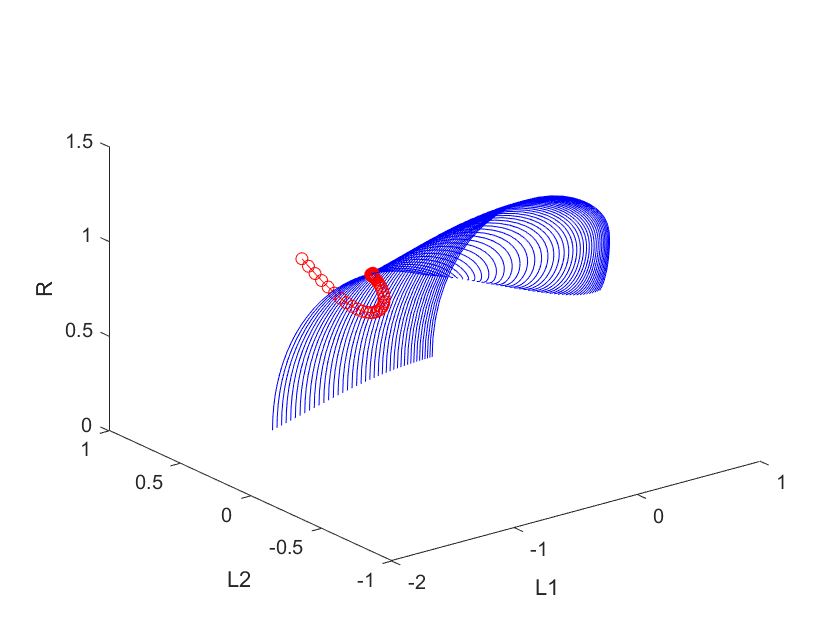} \hspace{1cm}
    \includegraphics[width=5cm]{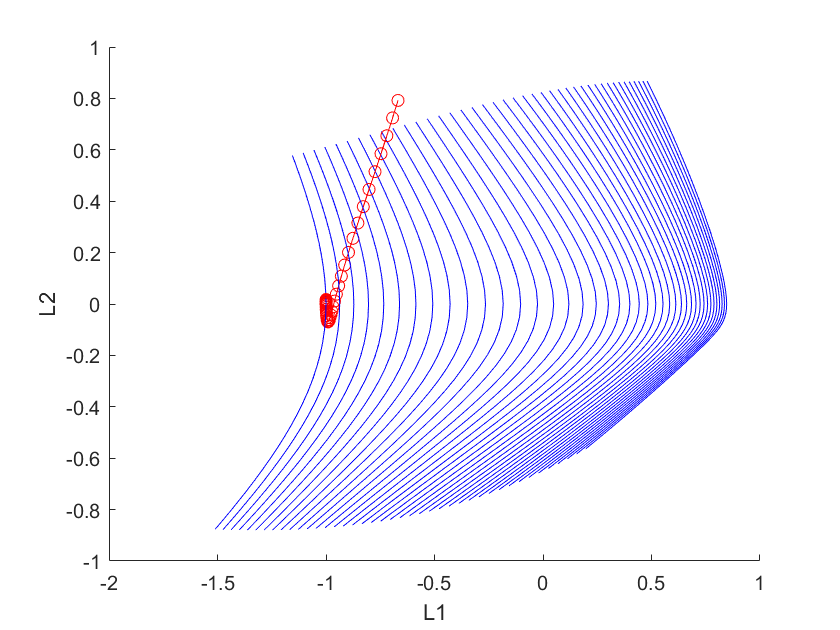}\\
    \vspace{1cm}
    \includegraphics[width=5cm]{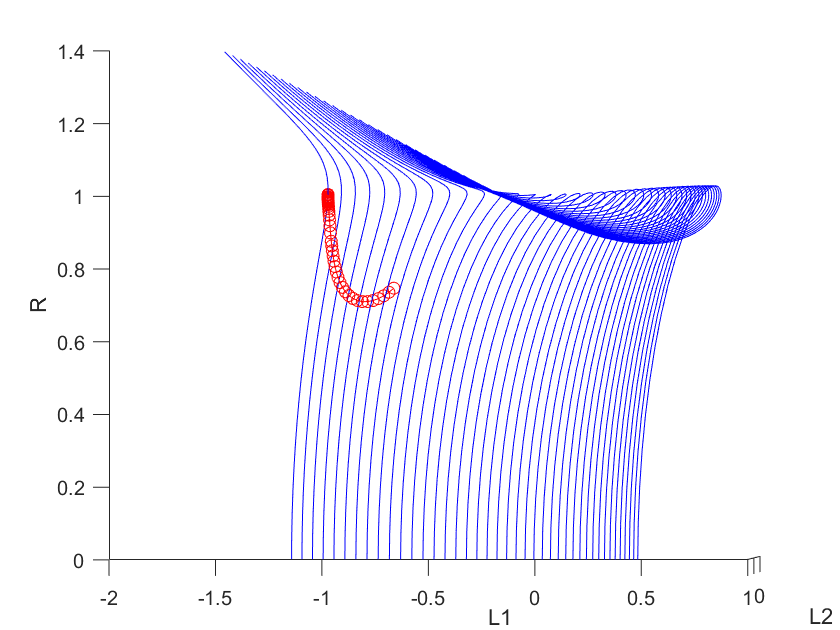}\hspace{1cm}
    \includegraphics[width=5cm]{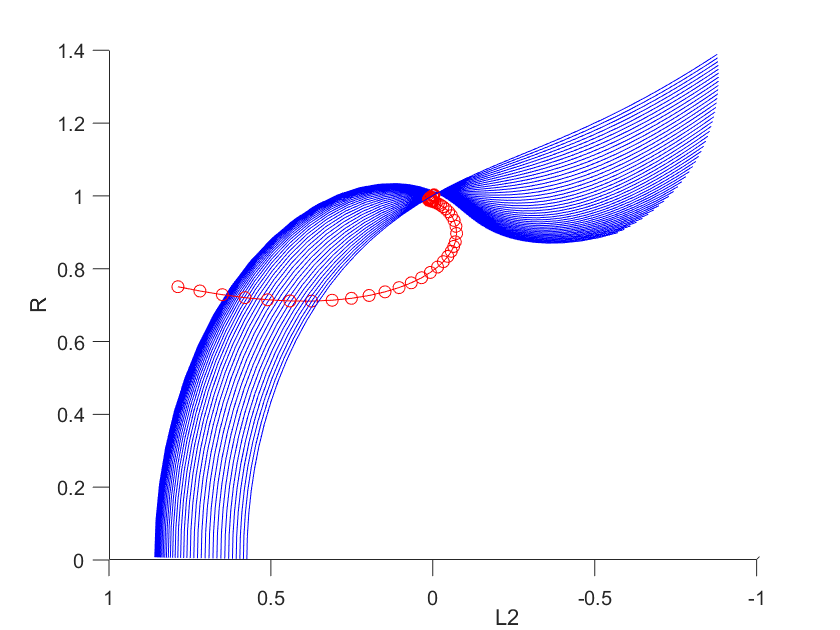}
\end{center}

Some important observations:
\begin{itemize}
 \item The intersection we see corresponds to the round sphere with 
 Einstein constant $1$, and is found by setting $\delta_1=\frac{1}{18}$, $\delta_2=-\frac{7}{9}$ and $\delta_3=\frac{1}{\sqrt{3}}$.
 \item As $\delta_1$ increases, 
 it appears the corresponding point on the curve $C$ approaches the point on the surface $S$ corresponding to $\delta_2=-1$ and $\delta_3=1$. This behaviour appears 
 to resemble that of the ancient `pancake' Ricci flow solution discussed in Section \ref{newancient}. 
\end{itemize}

\section{The surface $S$ close to $\delta_2=-1$, $\delta_3=1$}
The biggest difficulty in proving Theorem \ref{CT} is establishing an \textit{a priori} upper bound for $\delta_1$. This part of the proof 
is achieved by showing that any Ricci soliton that occurs with $\delta_1$ too large must have non-negative Riemann curvature everywhere. 
This is a two-step process: 
the first step is showing that if $\delta_1$ is too large, then the Riemann curvature must be non-negative between the $\mathbb{S}^1$ singular orbit
and the unique orbit with $\xi=10$, and the data at this 
point resembles the Gaussian soliton on $\mathbb{R}^2\times \mathbb{S}^2$; the second 
shows that if the data at $\xi=10$ is close to Gaussian, then the Riemann curvature must also be non-negative between the $\xi=10$ orbit 
and the $\mathbb{S}^2$ singular orbit. 
This section is devoted to the proof of Theorem \ref{S2orbit} below, which achieves the second step. Indeed, Theorem \ref{S2orbit} 
explicitly describes just how Gaussian we need to be at the $\xi=10$ principal orbit to guarantee 
curvature non-negativity between this orbit and the $\mathbb{S}^2$ orbit. The proof essentially involves an analysis of the behaviour of the soliton equations 
for values of $\delta_2$ and $\delta_3$ close to $-1$ and $1$ respectively. 
\begin{thm}\label{S2orbit}
 Suppose we have an $SO(2)\times SO(3)$-invariant gradient shrinking Ricci soliton on $\mathbb{S}^4$ of the form \eqref{metricform}. 
 Let $B_2= 10^{-146}$, and consider 
 the quantities $\xi,L_1,L_2,R$ associated to the soliton which satisfy \eqref{equationsnewf}. 
 If there is a principal orbit with $\xi=10$,
 $\left|L_1+\frac{1}{5+\sqrt{26}}\right|\le B_2$, $\left|R-1\right|\le B_2$, $\left|L_2\right|\le B_2$, and with $\mathcal{R}\ge 0$, 
 then $\mathcal{R}\ge 0$ between this principal orbit and the $\mathbb{S}^2$ orbit ($\mathcal{R}$ is described in \eqref{riemanncurvatureoperator}).
\end{thm}
\begin{rmk}
The value of $-\frac{1}{5+\sqrt{26}}$ in the statement of the above theorem arises in relation to the Gaussian soliton on $\mathbb{R}^2\times \mathbb{S}^2$. Indeed, this 
soliton is the special solution coming out of the $\mathbb{S}^2$ orbit with $-\xi(T-t)=\frac{1}{t}-t$, $-L_1(T-t)=\frac{1}{t}$, 
 $L_2(T-t)=0$, $R(T-t)=1$; when $\xi=10$, we have $L_1=-\frac{1}{5+\sqrt{26}}$. 
\end{rmk}

The proof of Theorem \ref{S2orbit} follows from the three lemmas presented below. The first lemma sets the goal posts, in that it tells us that curvature positivity 
between the $\xi=10$ orbit and the $\mathbb{S}^2$ orbit is guaranteed if curvature is positive at $\xi=10$ and $(\delta_2,\delta_3)$ is close to $(-1,1)$.  
\begin{lem}\label{S21}
 Suppose we have a soliton so that $(\delta_2,\delta_3)\in [-1,-1+B_1]\times [1-B_1,1+B_1]$, where $B_1=10^{-70}$.
 Then the corresponding solutions of the IVP in Proposition \ref{mapb} are such that the sectional curvatures
 $-L_1L_2$ and $\xi L_2+1-R^2-L_2^2$ do not change sign on $[\xi^{-1}(10),T)$. 
\end{lem}
The curvature positivity condition of the previous lemma is assumed in the hypothesis of Theorem \ref{S2orbit}, so we turn attention to 
the task of ensuring that $\max\{\left|\delta_3-1\right|,\delta_2+1\}\le B_1$ by having the soliton at the $\xi=10$ orbit quite close to the Gaussian soliton. 
\begin{lem}\label{S22}
 A Ricci soliton with $\left|R-1\right|,\left|L_2\right|\le B_2$ at the $\xi^{-1}(10)$ orbit must have 
 $\max\{\left|R(t)-1\right|,\left|L_2(t)\right|\}\le B_1$ for all $t\in (\xi^{-1}(10),T)$. In particular, $\left|\delta_3-1\right|\le B_1$. 
\end{lem}
\begin{lem}\label{S23}
 A Ricci soliton with $\left|L_1+\frac{1}{5+\sqrt{26}}\right|\le B_2$ at time 
 $\xi^{-1}(10)$ and $\max\{\left|R(t)-1\right|,\left|L_2(t)\right|\}\le B_1$ for all $t\in (\xi^{-1}(10),T)$ must have 
 $0\le \delta_2+1\le B_1$. 
\end{lem}
\begin{proof}[Proof of Theorem \ref{S2orbit}]
Using the hypothesis of Theorem \ref{S2orbit}, Lemma \ref{S22} implies that $\left|\delta_3-1\right|\le B_1$ and 
$$\max\{\left|R(t)-1\right|,\left|L_2(t)\right|\}\le B_1$$ for all $t\in (\xi^{-1}(10),T)$. Combining with the hypothesis of 
Theorem \ref{S2orbit}, Lemma \ref{S23} implies that $0\le \delta_2+1\le B_1$. Lemma \ref{S21} then implies that two of the Riemann curvature eigenvalues do not change sign. Since they are non-negative at time $\xi^{-1}(10)$, we obtain that these curvatures are non-negative on $(\xi^{-1}(10),T)$. 
We already know from Propositions 
\ref{scf1} and \ref{sf2p} that the other two Riemann curvature eigenvalues are non-negative. 
\end{proof}
We conclude this section with the proof of the three lemmas. 
\begin{proof}[Proof of Lemma \ref{S21}]
The strategy behind the proof of this lemma is simple enough: check the signs of $-L_1L_2$ and $\xi L_2+1-R^2-L_2^2$ using a Taylor series approximation, and 
show that the error of such an approximation is small enough. The result is obvious if $\delta_3=1$, because then $L_2=0$ and $R=1$ uniformly. Otherwise, consider
the new functions 
$w(t)=-\xi(T-t)-\frac{1}{t}$, $x(t)=-L_1(T-t)+\xi(T-t)$, $y(t)=\frac{-L_2(T-t)}{\delta_3-1}$, $z(t)=\frac{R(T-t)-1}{\delta_3-1}$
so that the vector $u=(w,x,y,z)$ satisfies 
\begin{align}\label{feu}
 u'+\frac{A}{t}u=Bu+C(u)+D
\end{align}
where 
\begin{align*}
  A=\begin{pmatrix}
     2&2&0&0\\
     0&-1&0&0\\
     0&0&1&0\\
     0&0&0&0
    \end{pmatrix},\
    B=\begin{pmatrix}
     0&0&0&0\\
     0&0&0&0\\
     0&0&0&2\\
     0&0&-1&0
    \end{pmatrix},\
    C(u)=\begin{pmatrix}
            -(x+w)^2-2(\delta_3-1)^2y^2\\
            x^2+xw+2(\delta_3-1)^2 y^2\\
            -wy+(\delta_3-1)z^2\\
            -(\delta_3-1)yz
           \end{pmatrix},\
           D=\begin{pmatrix}
            -1\\
            0\\
            0\\
            0
           \end{pmatrix},
    \end{align*}
and we have 
\begin{align*}
 u(0)=(0,0,0,1), \qquad u'(0)=(\delta_2,-\frac{3\delta_2+1}{2},\frac{\delta_3+1}{2},0), \qquad u''(0)=(0,0,0,-\frac{\delta_3(\delta_3+1)}{2}).
\end{align*}

Let $u_{2}=u(0)+u'(0)t+\frac{u''(0)t^2}{2}$ be the second-order Taylor series approximation to the solution, 
so that 
\begin{align*}
 u_2'+\frac{A}{t}u_2-Bu_2-C(u_2)&=D+\begin{pmatrix}
                                     \frac{(\delta_2+1)^2}{4}+\frac{(\delta_3-1)^2(\delta_3+1)^2}{2}\\
                                     -\frac{(3\delta_2+1)(\delta_2+1)}{4}-\frac{(\delta_3-1)^2(\delta_3+1)^2}{2}\\
                                     \frac{(\delta_3+1)(\delta_3+\delta_2)}{2}+\frac{(\delta_3-1)\delta_3(\delta_3+1)}{2}\\
                                     0
                                    \end{pmatrix}t^2+
                                    \begin{pmatrix}
                                     0\\
                                     0\\
                                    0\\
                                     \frac{(\delta_3-1)(\delta_3+1)^2\delta_3}{8}
                                    \end{pmatrix}t^3  
                                    +
                                    \begin{pmatrix}
                                     0\\
                                     0\\
                                    \frac{(1-\delta_3)\delta_3^2(1+\delta_3)^2}{16}\\
                                     0
                                    \end{pmatrix}t^4 
\end{align*}
Therefore, the function $v=u-u_2$ satisfies 
\begin{align*}
 v'+\frac{A}{t}v=Bv+F(t)v+C(v)+E_2t^2+E_3t^3+E_4t^4,\\
 v(0)=0, \qquad v'(0)=0, \qquad v''(0)=0,
\end{align*}
where 
\begin{align*}
 \left|E_2\right|_{\infty}\le 4 \max\{\left|\delta_2+1\right|,\left|\delta_3-1\right|\},\qquad
 \left|E_3\right|_{\infty}\le \max\{\left|\delta_2+1\right|,\left|\delta_3-1\right|\},\qquad
 \left|E_4\right|_{\infty}\le \frac{1}{3} \max\{\left|\delta_2+1\right|,\left|\delta_3-1\right|\},
\end{align*}
 provided $\max\{\left|\delta_2+1\right|,\left|\delta_3-1\right|\}\le 10^{-6}$, and
\begin{align*}
 F(t)&=\begin{pmatrix}
  0&0&0&0\\
  0&0&0&0\\
  0&0&0&2(\delta_3-1)\\
  0&0&-(\delta_3-1)&0
 \end{pmatrix}\\
 &+\begin{pmatrix}
  1+\delta_2&1+\delta_2&-2(\delta_3-1)^2(\delta_3+1)&0\\
  -\frac{3\delta_2+1}{2}&-2\delta_2-1&2(\delta_3-1)^2(\delta_3+1)&0\\
  -\frac{\delta_3+1}{2}&0&-\delta_2&0\\
  0&0&0&\frac{-(\delta_3-1)(\delta_3+1)}{2}
 \end{pmatrix}t+
 \begin{pmatrix}
  0&0&0&0\\
  0&0&0&0\\
  0&0&0&-(\delta_3-1)(\frac{\delta_3+1}{2})\delta_3\\
  0&0&(\delta_3-1)(\frac{\delta_3+1}{4})\delta_3&0
 \end{pmatrix}t^2.
\end{align*}

Therefore, 
\begin{align*}
 \left|v'(t)\right|_{\infty}\le \left|v(t)\right|_{\infty}\left(2.1t+\frac{1}{t}+2.1\right)+\max\{\left|\delta_2+1\right|,\left|\delta_3-1\right|\}\left(4t^2+t^3+\frac{1}{3}t^4\right),
\end{align*}
provided $\left|v\right|_{\infty}\le \frac{1}{100}$, and $\max\{\left|\delta_2+1\right|,\left|\delta_3-1\right|\}\le 10^{-6}$. 
We then find that, for $t\in [0,11]$, 
\begin{align}\label{mainvestimate}
 \frac{\left|v(t)\right|_{\infty}}{t}\le 10^{65}\times \max\{\left|\delta_2+1\right|,\left|\delta_3-1\right|\}\le 10^{-5}, 
\end{align}
since $\max\{\left|\delta_2+1\right|,\left|\delta_3-1\right|\}\le B_1$. 
The estimate \eqref{mainvestimate} gives $y(t)\ge \frac{\delta_3+1}{2}t-\left| v(t)\right|\ge 0$ on $[0,11]$, so $L_2(T-t)$ does not change sign for $t\in [0,11]$. We also find 
that, at the principal orbit $T-t$, 
\begin{align*}
 \frac{\xi L_2+1-R^2-L_2^2}{\delta_3-1}=-\left(\frac{1}{t}+w(t)\right)y(t)-z(t)(2+(\delta_3-1)z(t))-(\delta_3-1)y(t)^2,
\end{align*}
which is sufficiently close to $-2$ for $t\in [0,11]$ by \eqref{mainvestimate}. 
Therefore, these two sectional curvatures do not change sign on $[T-11,T)$. Since $\max\{\left|\delta_2+1\right|,\left|\delta_3-1\right|\}\le B_1$, we can also 
use these estimates to conclude that $-\xi(T-t)$ is close to $\frac{1}{t}-t$, so $\xi^{-1}(10)\in [T-11,T)$, as required.

\end{proof}

\begin{proof}[Proof of Lemma \ref{S22}]
A key quantity to consider is $K(t)^2=L_2(t)^2+(R(t)-1)^2$ because $K'(t)=\frac{-\xi L_2^2+L_2(R-1)}{K(t)}$, so that
\begin{align}\label{L2REvo}
\begin{split}
K(t)\left(-\frac{1}{2}-\max\{0,\xi\}\right) \le K'(t)\le \left(\frac{1}{2}+\max\{0,-\xi\}\right)K(t),\\
K(\xi^{-1}(10))\le \sqrt{2} \delta. 
\end{split}
\end{align}
 Using \eqref{L2REvo} and the fact that $\xi^{-1}(0)-\xi^{-1}(10)\le 10$ (since $\xi'\le -1$), 
we find that 
\begin{align}\label{kestimatesearly}
 K(t)\le \sqrt{2}e^{5}B_2, \ t\in (\xi^{-1}(10),\xi^{-1}(0)].
\end{align}
Now let $t^*\in (\xi^{-1}(0),T-\frac{B_1}{2}]$ be the last time so that 
\begin{align}\label{XE}
 -\xi(t)\le -L_1(t)+\frac{1}{2}\le \frac{1}{T-t}+\frac{1}{2} \ \text{for all}\  t\in [\xi^{-1}(0),t^*). 
\end{align}
Such a $t^*$ must exist because $0\le -L_1\le \frac{1}{T-t}$ everywhere by Proposition \ref{scf1}, so the estimate holds at time $\xi^{-1}(0)$. 
Then using \eqref{L2REvo} and \eqref{XE}, we obtain that $K'(t)\le \left(1+\frac{1}{T-t}\right)K(t)$; integrating from $\xi^{-1}(0)$ to $t^*$ gives 
\begin{align}\label{kestimateotstar}
 K(t)\le \frac{\sqrt{8}B_2e^{T-\xi^{-1}(0)+5}(T-\xi^{-1}(0))}{B_1}, \ t\in (\xi^{-1}(0),t^*).
\end{align}
Now, using the fact that $y'\ge y^2+1$ and $y(\xi^{-1}(0))\ge 0$, where $y=\min\{-\xi,-L_1\}$, we find that $T-\xi^{-1}(0)\le \frac{\pi}{2}$. 
Combining this with \eqref{kestimateotstar} and the definition of $B_1$, we obtain that $K(t)\le \frac{B_1}{10}$ for all $t\in (\xi^{-1}(0),t^*)$.
This implies that, for $t\in (\xi^{-1}(0),t^*)$, we have 
\begin{align*}
 (-\xi+L_1)'&=L_1(-\xi+L_1)+2L_2^2\\
 &\le L_1(-\xi+L_1)+10^{-4}.
\end{align*}
Since $-\xi+L_1$ is negative at time $\xi^{-1}(0)$, and $L_1$ is negative everywhere, we 
find that $-\xi(t)< -L_1(t)+\frac{1}{2}$ for $t\in (\xi^{-1}(0),t^*)$, so $t^*=T-\frac{B_1}{2}$. 

Combining $K(T-\frac{B_1}{2})\le \frac{B_1}{10}$ with $\left|R'\right|\le R^2$ then gives 
\begin{align}\label{efrcts2}
 \left|R(t)-1\right|\le B_1 \  \text{for} \ t\in (T-\frac{B_1}{2},T).
\end{align}
Equation \eqref{efrcts2} combined with \eqref{kestimateotstar} and \eqref{kestimatesearly} then gives
\begin{align}\label{efrcts3}
 \left|R(t)-1\right|\le B_1 \  \text{for} \ t\in (\xi^{-1}(10),T).
\end{align} In particular, we find that $\left|\delta_3-1\right|\le B_1$.

To conclude the proof, it therefore suffices to check the smallness of $L_2$ on $[\xi^{-1}(10),T)$. The required bound for $L_2$ 
on $[\xi^{-1}(10),T-\frac{B_1}{2}]$ follows
from \eqref{kestimatesearly}, \eqref{kestimateotstar} and the fact that $t^*=T-\frac{B_1}{2}$. The required $L_2$ bound on $[T-\frac{B_1}{2},T)$ 
follows from the estimate $\left|L_2(t)\right|'\ge -\left|R(t)^2-1\right|\ge -B_1(R(t)+1)\ge -\frac{5}{2}B_1$ for 
$t\in (\xi^{-1}(0),T)$ coupled with $L_2(T)=0$. 
\end{proof}

\begin{proof}[Proof of Lemma \ref{S23}]
The estimate $\left|R(t)-1\right|\le B_1$ implies that $\left|L_2(t)\right|'\ge -B_1(R(t)+1)\ge -\frac{5}{2}B_1$ for 
$t\in (\xi^{-1}(0),T)$. Therefore, 
$\left|L_2(t)\right|\le \frac{5B_1(T-t)}{2}$ for all $t\in (\xi^{-1}(0),T)$, because $L_2(T)=0$. In fact, we actually obtain 
\begin{align}\label{L240e}
 \left|L_2(t)\right|\le \frac{5B_1 (T-t)}{2} \ \text{for all} \ t\in (\xi^{-1}(10),T)
\end{align}
because $\left|L_2(t)\right|\le B_1$ on $(\xi^{-1}(10),\xi^{-1}(0))$ by Lemma \ref{S22}, and $T-\xi^{-1}(0)>\frac{2}{5}$. This 
estimate on $T-\xi^{-1}(0)$ follows from the fact that $L_2^2+(R-1)^2\le 2B_1^2$ is small at time $\xi^{-1}(0)$, and $L_1(\xi^{-1}(0))\in [-1,0]$ 
(since $\xi L_1+1-L_1^2\ge 0$ everywhere). 

Consider the function $x(t)=-\xi(T-t)+L_1(T-t)+t$. In light of \eqref{L240e}, the function $x$ satisfies 
\begin{align*}
 x'(t)\ge -25B_1^2t^2+\frac{x}{t}, \qquad x'(0)=\frac{3}{2}(\delta_2+1), 
\end{align*}
for all $t$ with $x(t)\le t$. 
If $\delta_2+1>B_1$, we have that $x(t)>\frac{B_1 t}{2}$ for all $t\in (0,12)$. 
Since $0\le -L_1(T-t)\le \frac{1}{t}$, we then find that 
\begin{align}\label{finalestimatexB1}
 -\xi(T-t)+L_1(T-t)-\frac{1}{L_1(T-t)}\ge x(t)>\frac{B_1t}{2}.
\end{align}
Now, the smallness of 
$(R(\xi^{-1}(0))-1)^2+L_2(\xi^{-1}(0))^2$ and the closeness of $L_1(\xi^{-1}(0))$ to $-\frac{1}{5+\sqrt{26}}$ implies that 
$T-\xi^{-1}(10)\ge \frac{1}{2}$. Therefore, \eqref{finalestimatexB1} is in contradiction with
$\left|L_1(T-t)+\frac{1}{5+\sqrt{26}}\right|\le B_2$ 
at time $t=T-\xi^{-1}(10)$. 
\end{proof}

\section{The curve C for large values of $\delta_1$}
In our quest to prove that normalised $SO(2)\times SO(3)$-invariant solitons on $\mathbb{R}^4$ have uniformly bounded Riemann curvature and volume (as well as a
uniform lower bound on the injectivity radius), it is necessary 
to find uniform bounds for the $\delta_1,\delta_2,\delta_3$ numbers discussed in Propositions \ref{mapa} and \ref{mapb}. This section is 
dedicated to the proof of Theorem \ref{d1large} below, which combines with Theorem \ref{S2orbit} above to produce the required bound for $\delta_1$ (recall that a Ricci soliton on 
$\mathbb{S}^4$ with non-negatve Riemann curvature operator must be round). 

\begin{thm}\label{d1large}
Suppose we have an $SO(2)\times SO(3)$-invariant gradient shrinking Ricci soliton on $\mathbb{S}^4$ of the form \eqref{metricform}. 
 Consider the quantities $\xi,L_1,L_2,R$ associated to the soliton which satisfy \eqref{equationsnewf}, and suppose $\delta \in (0,10^{-25})$ and 
 $\delta_1>\frac{1}{\delta^{120}}$. Then there is a principal orbit with $\xi=10$,
 $\left|L_1+\frac{1}{5+\sqrt{26}}\right|\le  \delta$, $\left|R-1\right|\le \delta$, $\left|L_2\right|\le \delta$, 
 and so that the Riemann curvature is non-negative between this 
 principal orbit and the $\mathbb{S}^1$ orbit. 
\end{thm}
When proving this theorem, it is useful to again change variables with $w=L_1$, $x=\frac{L_2}{R}$, $y=\frac{R}{\xi}$, $z=\frac{1}{R}$ so that 
\begin{align}\label{firstchangebeforescaling}
\begin{split}
 w'&=\frac{-w-yz}{yz},\\
 x'&=\frac{-x+y-yz^2+x^2y}{yz}, \\
 y'&=\frac{-xy^2+y^3(w^2z^2+2x^2+z^2)}{yz},\\
 z'&=\frac{xyz}{yz}.
 \end{split}
\end{align}
Therefore our study essentially reduces to the analysis of the integral curves of the following system of equations:
\begin{align}\label{FirstChange}
\begin{split}
 w'&=-w-yz,\\
 x'&=-x+y-yz^2+x^2y, \\
 y'&=-xy^2+y^3(w^2z^2+2x^2+z^2),\\
 z'&=xyz.
 \end{split}
\end{align}
In these co-ordinates, the non-negativity of $-L_1L_2$ and $\xi L_2+1-R^2-L_2^2$ is implied by the non-negativity 
of $x$ and $\frac{x}{y}+z^2-1-x^2$ (since we already know that $L_1\le 0$). The solutions we are interested in are part of the two-dimensional unstable manifold of the critical point of \eqref{FirstChange} at 
$(0,1,\frac{1}{2},0)$; the $\delta_1$ parameter now describes the initial direction of travel through this two-dimensional unstable manifold, and is given 
precisely by $\delta_1=\lim_{t\to 0}\frac{\frac{1}{z(t)}-\frac{1}{t}}{t}$ for the solutions of \eqref{firstchangebeforescaling}. 

In this section, we are only interested in evolving the system \eqref{firstchangebeforescaling} up until the time $\xi^{-1}(10)$, so we can assume that $y\in (0,\infty)$. It is handy to keep track of the 
evolution of the quantities $C=\frac{x}{y(1-z)}$, $D=\frac{w}{y}-w^2$ and $E=\frac{x}{y}+z^2-1-x^2$. Indeed, the closeness of $D$ to $-1$ measures the Gaussian 
structure, positivity of $E$ ensures positivity of the last sectional curvature, and $C$ is an analogue of a quantity that arises in the construction of the Bryant soliton. We find that 
\begin{align}\label{usefulquantities}
\begin{split}
 C'&=-C+(1+z)+Cxy+C(xy-y^2(w^2z^2+2x^2+z^2))+C^2y^2z\\
 D'&=(-D-1)(1-yw)+\tilde{D}\\
 E'&=xy\left(1-w^2z^2-x^2\right)+\tilde{E},
 \end{split}
\end{align}
where $\left|\tilde{D}\right|\le 3\left|z-1\right|+3\left|x\right|$ whenever $0\le x,y,-w,z\le 1$, 
and $\tilde{E}=0$ whenever $E=0$ and $y\in (0,\infty)$. 
Therefore, to prove Theorem \ref{d1large}, it suffices to prove the following:
\begin{thm}\label{delta1largesimple}
 Choose $\delta\in (0,10^{-25})$. If we have a Ricci soliton with $\delta_1\ge \frac{1}{\delta^{120}}$, 
 then the corresponding trajectory of \eqref{FirstChange} which $\alpha$-limits to $(0,1,\frac{1}{2},0)$ includes a point at 
 which $yz=\frac{1}{10}$, $ 1-\frac{\delta}{3}\le z\le 1$, $0\le -w\le 1$, $\left|D+1\right|\le \frac{\delta}{3}$ and $x\le \frac{\delta}{3}$, and so that $\min\{x,E\}\ge 0$ on the trajectory up until 
 that point. 
\end{thm}
To proof of this theorem essentially follows by noting the behaviour of the unstable trajectory coming out of $(0,1,\frac{1}{2},0)$ as $\delta_1$ 
becomes large:
\begin{itemize}
 \item the trajectory travels from $(0,1,\frac{1}{2},0)$, and gets close to the critical point $(0,0,0,0)$;
 \item the trajectory then travels close to the line of critical points $(0,0,0,z)$ for $z\in [0,1]$;
 \item the trajectory then breaks away from this line of critical points as $y>0$ begins to grow while $z$ stays close to $1$, and we get a point satisfying the conclusion of 
 Theorem \ref{delta1largesimple}. 
\end{itemize}
To prove Theorem \ref{delta1largesimple}, we proceed working backwards, and start by determining how close the unstable trajectory 
needs to be to touching the point $(w,x,y,z)=(0,0,0,0)$ in order 
to get the desired conclusion. 
\begin{lem}\label{0tofinal}
 Choose $\delta\in (0,10^{-25})$ and suppose our trajectory of \eqref{FirstChange} includes a point with  
 $0\le -w,x,y,z\le \delta^6$, $\left|\frac{y}{x}-1\right|\le  \delta^6$, $\left|D\right|\le 1$ and $E\ge 0$. 
 Then the trajectory includes a later point at which $yz=\frac{1}{10}$, 
 $ 1-\frac{\delta}{3}\le z\le 1$, $0\le -w\le 1$, $\left|D+1\right|\le \frac{\delta}{3}$ and $x\le \frac{\delta}{3}$. 
 Furthermore, $\min\{E,x\}\ge 0$ on the trajectory between these two points. 
\end{lem}
\begin{proof}
We can assume that we have a solution of \eqref{FirstChange} so that $0$ is the time of the first point.

\textbf{Step One: construct an interval $(0,t^*)$ on which we expect to find the required terminal point; find some basic estimates.}
 Define $t^*>0$ so that 
 \begin{align}\label{definitiontstart}
  (0,t^*) \ \textit{is the maximal time interval on which} \ \left|x\right|\le \frac{1}{10}, \left|y\right|\le \frac{1}{9}, \left|w\right|\le \frac{1}{2} \ \textit{and} \ 
  z\in (0,1). 
 \end{align}
 On such an interval, \eqref{usefulquantities} implies that 
 \begin{align*}
  -C+1-\frac{C}{10}\le C'\le -C+2+\frac{C^2}{50}+\frac{C}{10},
 \end{align*}
 so since $\left|C(0)-1\right|<10^{-3}$, we find that $C(t)\in[0.9,2.4]$ for all $t\in (0,t^*)$. 
The equations for $y$ and $z$ in \eqref{FirstChange} can then be written 
\begin{align}\label{yztmid}
\begin{split}
\frac{y'}{y^2}&=y\left(-C(1-z)+w^2z^2+2C^2y^2(1-z)^2+z^2\right),\\
 \frac{z'}{y^2}&=Cz(1-z).
 \end{split}
\end{align}
Now we claim that $t^*>z^{-1}(\frac{1}{2})$. Indeed, \eqref{yztmid} and the fact that 
$y(0)\le \delta^6$ with $C\in [0.9,2.4]$ implies that 
\begin{align}\label{estimateforyoninterval}
 y(t)\le \delta^6 \ \text{for} \ t\in (0,\min\{t^*,z^{-1}(\frac{1}{2})\}). 
\end{align}
Estimate \eqref{estimateforyoninterval}, coupled with the equation for $w'$ in \eqref{FirstChange} and the definition of $C$ implies that
\begin{align}\label{estimateforxwoninterval}
 \left|x\right|=\frac{Cy}{\left|1-z\right|}\le 5 \delta^6 \ \text{and} \ \left|w\right|<\frac{1}{4}\ \text{for} \ t\in (0,\min\{t^*,z^{-1}(\frac{1}{2})\}). 
\end{align}
The estimates \eqref{estimateforyoninterval} and \eqref{estimateforxwoninterval} combine with the definition of $t^*$ in \eqref{definitiontstart} 
to show that $t^*>z^{-1}(\frac{1}{2})$. 

Now, on $[z^{-1}(\frac{1}{2}),t^*]$, \eqref{yztmid} gives us $\frac{y'}{y^2}\le 1.5y$ 
and $\frac{(1-z)'}{y^2}\le -0.4(1-z)$.
Therefore, $y(t)\le \delta^6e^{1.5(s-z^{-1}(\frac{1}{2}))}$ and $(1-z(t))\le \frac{1}{2} e^{-0.4(s-z^{-1}(\frac{1}{2}))}$, where $s$ is some function of $t$. 
Combining these estimates gives 
\begin{align}\label{yzestimate}
 y(t)(1-z(t))^{3.75}\le \frac{\delta^6}{10} \ \text{for all} \ t\in  [z^{-1}(\frac{1}{2}),t^*].
\end{align}
We conclude this step with the assertion that $y(t^*)=\frac{1}{9}$. We show this by demonstrating that all of the other inequalities defining $t^*$ in 
\eqref{definitiontstart} are strict at 
time $t^*$ itself. 
On the interval $(0,t^*)$, 
$yz< \frac{1}{4}$, so $-w<\frac{1}{4}$ on $(0,t^*)$ since $-w(0)< \frac{1}{4}$. 
For any $t\in [0,t^*]$ with $(1-z)< \frac{9}{2.4}\delta$, it is clear that $x\le Cy(1-z)<\delta$. 
On the other hand, if $(1-z)\ge \frac{9}{2.4}\delta$, 
then $x\le 2.4 \frac{y(1-z)^{3.75}}{(1-z)^{2.75}}< \delta$ by \eqref{yzestimate}
as well. It is clear from \eqref{yztmid} that $z(t^*)<1$, so it must be the case that $y(t^*)=\frac{1}{9}$. 

\textbf{Step Two: show that $x$ and $E$ are non-negative on $(0,t^*)$.}
We have $x\ge 0$ on $(0,t^*)$ 
by \eqref{FirstChange}, since $z\in (0,1)$ on $(0,t^*)$ and $x(0)\ge 0$. The estimates on $x,z,w$ provided by the definition of $t^*$ in \eqref{definitiontstart} 
then imply that $1-w^2z^2-x^2\ge 0$ on 
$(0,t^*)$; this observation coupled with \eqref{usefulquantities} and $E(0)\ge 0$ implies that $E\ge 0$ here as well.

\textbf{Step Three: find the required point in  $(0,t^*)$.} 
Let $I$ be the connected interval of times ending at $t^*$ so that $y(t)\in [\frac{1}{10},\frac{1}{9}]$. 
We find from \eqref{yzestimate} that $(1-z)\le \delta^{\frac{6}{3.75}}<\frac{\delta}{200}$ in $I$.
The definition of $C$ then implies that $0\le x\le \frac{\delta}{3}$ in $I$ as well. 
The intermediate value theorem also gives us that there is some $t_1\in I$ at which $yz=\frac{1}{10}$.
Finally, we need to demonstrate that $\left|D+1\right|\le \frac{\delta}{3}$ at $t_1$. 
At time $(1-z)^{-1}(\frac{\delta}{100})$ (which is less than $t_1$ by the above estimate for $(1-z)$ on $I$), \eqref{yzestimate} tells us that  
$y\le \frac{\delta^6100^{3.75}}{10\delta^{3.75}}\le \delta^{1.25}$. 
The fact that $y'\le 1.5y^3$ on $[z^{-1}(\frac{1}{2}),t^*]$ implies that the distance between $I$ and $(1-z)^{-1}(\frac{\delta}{100})$ is greater 
than $\frac{1}{100\delta^2}$. 
This large amount of time tells us that $D$ must get quite close to $-1$ by the time we land in $I$. 
Indeed, from \eqref{usefulquantities} and \eqref{yzestimate} we find that 
\begin{align}\label{Dfine}
\left|D+1\right|'\le -\frac{3}{4}\left|D-1\right|+\frac{\delta}{10}
\end{align}
 on $((1-z)^{-1}(\frac{\delta}{100}),t^*)$, while 
 \begin{align}\label{Dcrude}
  \left|D+1\right|'\le -\frac{3}{4}\left|D-1\right|+10
 \end{align}
 holds on $(0,t^*)$. Estimate \eqref{Dfine} tells us that whenever $\left|D+1\right|\ge \frac{\delta}{3}$ and 
 $t\in((1-z)^{-1}(\frac{\delta}{100}),t^*)$ , we have $\left|D+1\right|'\le -\frac{\delta}{10}$. 
 The estimate \eqref{Dcrude} coupled with $\left|D(0)\right|\le 1$ 
 implies that $\left|D\right|\le 15$ at time $(1-z)^{-1}(\frac{\delta}{100})$, so 
 it takes no more than time $\frac{10^5}{\delta}>\frac{1}{100\delta^2}$ to get $\left|D+1\right|\le \frac{\delta}{3}$. Therefore, $\left|D+1\right|\le \frac{\delta}{3}$ 
 is achieved for all times in $I$ (including $t_1$). 
\end{proof}
Now we discuss how making $\delta_1$ large can force our trajectory to be close to the $(0,0,0,0)$ point, in the sense of Lemma \ref{0tofinal}. 
This is achieved with the two lemmas below. 
\begin{lem}
Let $B_3\in (0,10^{-100})$. Suppose our trajectory solution of \eqref{FirstChange} has a point so that 
 $0\le z\le B_3^{12}$, $0\le -w\le B_3$, $E\ge 0$ and 
 $$\{\left|x-x^{(\infty)}(\frac{1}{9})\right|,\left|y-y^{(\infty)}(\frac{1}{9})\right|\}\le B_3^6,$$
 where $(x^{(\infty)},y^{(\infty)})$ is the special Bryant soliton solution discussed in Theorem \ref{BryantSolitonsmalltime} of the Appendix.  
 Then there is a later point on the trajectory at which $0\le -w,x,y,z\le B_3$, $\left|\frac{y}{x}-1\right|\le B_3$, $\left|D\right|\le 1$, 
 and so that $\min\{E,x\}\ge 0$ on the trajectory between these two points.
\end{lem}

\begin{proof}
This time, we find it convenient to have our $(w,x,y,z)$ solution of \eqref{FirstChange} so that the initial point 
described in the hypothesis of the lemma corresponds to time $\frac{1}{9}$. This is so we can easily compare our solution to the special 
$x^{(\infty)},y^{(\infty)}$ solution. Strictly speaking, these functions $x^{\infty},y^{\infty}$ solve \eqref{bryantfirstchange}, but we now reparametrise 
them so that they solve \eqref{xysystmebryant} instead, with a parametrisation that preserves $\frac{1}{9}$. 

\textbf{Step One: estimates on the limiting solution.}
As discussed in Appendix \ref{BryantAppendix}, we have 
\begin{align}\label{BryantEstimates}
 (x^{(\infty)})'(t)\le 0, \qquad (y^{(\infty)})'(t)\le 0, \qquad \frac{y^{(\infty)}(t)}{x^{(\infty)}(t)}\in [\frac{1}{2},1] \ \text{for all} \ t\in [\frac{1}{9},\infty).
\end{align}
Combining \eqref{BryantEstimates} with \eqref{xysystmebryant} gives 
\begin{align}\label{bryanty'estimates}
 (y^{(\infty)})'\le -(y^{(\infty)})^3(1-4(y^{(\infty)})^2).  
\end{align}
Also note that \eqref{BryantEstimates}, combined with Theorem \ref{BryantSolitonsmalltime} and Proposition \ref{initialestaimtesxy} in Appendix A implies that 
\begin{align}\label{initialyestimatebryant}
\begin{split}
1-0.0375\le x^{(\infty)}(\frac{1}{9})\le 1-0.0331\\
\frac{1}{2}-0.01875\le y^{(\infty)}(\frac{1}{9})\le \frac{1}{2}-0.01545. 
\end{split}
\end{align}
It is handy to note that  
\begin{align*}
 [\frac{1}{9},\infty)\subseteq \overline{A^*\cup B^*\cup C^*\cup D^*}
\end{align*}
with 
\begin{align*}
 A^*=(y^{(\infty)})^{-1}(\frac{1}{3},\frac{1}{2}-0.01545), B^*=(y^{(\infty)})^{-1}(\frac{1}{8},\frac{1}{3}), 
 \ C^*= (y^{(\infty)})^{-1}(\frac{B_3}{2},\frac{1}{8}), \ 
 D^*=(y^{(\infty)})^{-1}(0,\frac{B_3}{2}). 
\end{align*}
By \eqref{bryanty'estimates}, we find that $(y^{(\infty)})'(t)\le -\frac{1}{27}(1-4(y^{(\infty)}(t))^2)$ on $A^*$ so 
that $\left|A^*\right|\le 18$. Also, $(y^{(\infty)})'(t)\le -\frac{1}{2\cdot 8^3}$ on $B^*$ so that $\left|B^*\right|\le 250$. Finally, 
$\left|C^*\right|\le \frac{3}{B_3^2}$, 
because $(y^{(\infty)})'(t)\le -\frac{15 (y^{(\infty)}(t))^3}{16}$ for $t\in C^*$. 

\textbf{Step Two: closeness to the limiting solution.}
Define the positive number $t_0\le \sup C\le 268+\frac{3}{B_3^2}$ so that 
\begin{align}\label{tnodefinition}
 [\frac{1}{9},t_0] \ \text{is the maximal time interval on which} \ z(t)\le B_3^6.
\end{align}
We will now obtain the required estimates by showing that the quantities 
$X=x-x^{(\infty)}$ and $Y=y-y^{(\infty)}$ are small on the interval $[\frac{1}{9},t_0]$. 
By examining \eqref{FirstChange} and \eqref{xysystmebryant}, we compute 
\begin{align}\label{bryantcomparisonestimates}
 \begin{pmatrix}
  X\\
  Y
 \end{pmatrix}'&=
 \begin{pmatrix}
  -1+2x^{(\infty)}y^{(\infty)}&1+(x^{(\infty)})^2\\
  (y^{(\infty)})^2(4x^{(\infty)}y^{(\infty)}-1)&x^{(\infty)}y^{(\infty)}(6x^{(\infty)}y^{(\infty)}-2)
 \end{pmatrix}\begin{pmatrix}
 X\\
 Y
 \end{pmatrix}
 +O(X^2+Y^2)+Z
\end{align}
where $\left|O(X^2+Y^2)\right|_{\infty}\le 100( \left|X\right|^2+\left|Y\right|^2)$ provided $\left|X\right|^2+\left|Y\right|^2\le 1$, 
and $\left|Z\right|_{\infty}\le 4\left|z\right|^2$
provided $x,y,-w,z\in [0,1.1]$. 
We now use \eqref{bryantcomparisonestimates} to obtain smallness of $(X,Y)$ on $(\frac{1}{9},t_0)$.  
Let $V=\max\{\left|Y\right|,2\left|X-Y\right|\}$; \eqref{bryantcomparisonestimates} implies that 
\begin{align*}
 V'&\le -\frac{B_3^2}{8} V+16 B_3^6
\end{align*}
on $C^*$, as long as $\left|V\right|\le B_3^3$, and 
\begin{align*}
 V'&\le 2V+16B_3^6
\end{align*}
on $A^*\cup B^*$, provided $\left|V\right|\le B_3^3$. 
Therefore, since $V(\frac{1}{9})\le  4B_3^6$, we find that 
\begin{align}\label{longtimevestimates}
 V(t)\le B_3^3 \ \text{for all} \ t\in \overline{A^*\cup B^*\cup C^*}\cap [\frac{1}{9},t_0]. 
\end{align}

\textbf{Step Three: concluding estimates.}
We need to check that $x,y,z,\frac{y}{x}-1,w$ are all smaller than $B_3$ at time $t_0$, that $\left|D\right|\le 1$, and that $E,x\ge 0$ on $[\frac{1}{9},t_0]$. 
We claim that $y^{(\infty)}(t_0)=\frac{B_3}{2}$. 
To see this, we use \eqref{FirstChange} to estimate 
\begin{align*}
 \frac{z'}{z}&=xy\\
 &\le (x^{(\infty)}+3B_3^3)(y^{(\infty)}+B_3^3)\\
 &\le 2(y^{(\infty)}+\frac{3B_3^3}{2})^2\le 2.5 (y^{(\infty)})^2
\end{align*}
on $[\frac{1}{9},t_0]$.
Therefore, $z(t)\le B_3^{12}e^{28}$ for all $t\in A^*$, so $t_0\notin \overline{A^*}$. 
On $B^*$, we estimate that $z'\le \frac{z}{2}$, so $z(t)\le B_3^{12}e^{153}\le B_3^{11}$ for all $t\in B^*$ and $t_0\notin \overline{B^*}$. 
On the other hand, \eqref{bryanty'estimates} gives $(y^{\infty})'\le -\frac{15(y^{(\infty)})^3}{16}$ on $C^*$, so that 
\begin{align*}
(y^{(\infty)}(t))^2\le \frac{8}{15(t-(y^{(\infty)})^{-1}(\frac{1}{8}))+512}
\end{align*}
for $t\in \overline{C^*}$. We consequently find 
\begin{align*}
 \frac{z'(t)}{z(t)}\le \frac{20}{15(t-(y^{(\infty)})^{-1}(\frac{1}{8}))+512},
\end{align*}
so that 
$z(t)<B_3^6$ for $t\in C^*$. Therefore, $y^{(\infty)}(t_0)=\frac{B_3}{2}$ as required. 

The smallness of $y^{(\infty)}(t_0)$ and \eqref{longtimevestimates} imply the required 
estimates for $x(t_0),y(t_0)$. The estimate for $z(t_0)$ follows from the definition of $t_0$. 
The estimates for $w$ and $D$ follow immediately from \eqref{longtimevestimates} 
and the equation for $w'$ in \eqref{FirstChange}. Using Proposition \ref{FinalEstff}, we also have
\begin{align*}
\left|\frac{y}{x}-1\right|&\le \left|\frac{y}{x}-\frac{y^{(\infty)}}{x^{(\infty)}}\right|+\left|\frac{y^{(\infty)}}{x^{(\infty)}}-1\right|\\
&\le \frac{\left|yx^{(\infty)}-xy^{(\infty)}\right|}{x^{(\infty)}(x^{(\infty)}-2V)}+\frac{3B_3}{4}\\
&\le B_3,
\end{align*}
 since $V\le B_3^3$, while $x^{(\infty)}\ge y^{(\infty)}=\frac{B_2}{2}$. 
It is clear that $x$ is non-negative on $[\frac{1}{9},t_0]$, since $x^{(\infty)}\ge \frac{B_3}{2}$ and $V\le  B_3^3$. To show that 
$E$ is non-negative on $[\frac{1}{9},t_0]$, note that $1-w^2z^2-x^2$ is non-negative on $[\frac{1}{9},t_0]$ since $x^{(\infty)}\in[\frac{B_3}{2}, 1-0.331]$, 
$V\le B_3^3$, $z\le B_3$ and $w\le B_3$ (by \eqref{FirstChange} and the estimates on $y$ and $w$). The non-negativity of $1-w^2z^2-x^2$ and the fact that 
$E(\frac{1}{9})\ge 0$ implies that $E\ge 0$ on $[\frac{1}{9},t_0]$ by \eqref{usefulquantities}. 

\end{proof}
\begin{lem}
 For each $B_3\in (0,10^{-100})$, choose an aribtrary $\delta_1\ge B_3^{-20}$. Then the solution of \eqref{firstchangebeforescaling} satisfies 
 $0\le z(\frac{1}{9\sqrt{\delta_1}})\le B_3^{12}$, $0\le -w(\frac{1}{9\sqrt{\delta_1}})\le B_3$, and 
 $\{\left|x(\frac{1}{9\sqrt{\delta_1}})-x^{(\infty)}(\frac{1}{9})\right|,\left|y(\frac{1}{9\sqrt{\delta_1}})-y^{(\infty)}(\frac{1}{9})\right|\}\le B_6^4$,
 where $(x^{(\infty)},y^{(\infty)})$ is the Bryant soliton solution discussed in the Appendix. Furthermore, $\min\{E,x\}\ge 0$ on $(0,\frac{1}{9\sqrt{\delta_1}})$. 
\end{lem}
\begin{proof}
Let $p=\frac{1}{\sqrt{\delta_1}}$, $t^*=\frac{1}{9}$, and consider the rescaled functions $\tilde{f}(t)=pf(pt)$ for $f=R,L_1,L_2,\xi$.
Then these new functions satisfy 
\begin{align*}
 \tilde{\xi}'&=-\tilde{L_1}^2-2\tilde{L_2}^2-p^2,\\
 \tilde{L_1}'&=-\tilde{\xi} \tilde{L_1}-p^2,\\
 \tilde{L_2}&=-\tilde{\xi} \tilde{L_2}+\tilde{R}^2-p^2,\\
 \tilde{R}'&=-\tilde{L_2} \tilde{R},
\end{align*}
with the `new' $\delta_1$ equal to $1$, i.e., $\lim_{t\to 0}\frac{\tilde{R}(t)-\frac{1}{t}}{t}=1$. 
The discussion of the Bryant soliton on $\mathbb{R}^3$ discussed in the Appendix implies the existence of
smooth functions $\xi^{(\infty)},L_1^{(\infty)},L_2^{(\infty)},R^{(\infty)}:(0,\infty)\to \mathbb{R}$ of the form \eqref{etaforma} so that 
\begin{align*}
(\xi^{(\infty)})'&=-(L_1^{(\infty)})^2-2(L_2^{(\infty)})^2,\\
(L_1^{(\infty)})'&=-\xi^{(\infty)}L_1^{(\infty)},\\
(L_2^{(\infty)})'&=-\xi^{(\infty)}L_2^{(\infty)}+(R^{(\infty)})^2,\\
 (R^{(\infty)})'&=-L_2^{(\infty)} R^{(\infty)},\\
1&= \lim_{t\to 0}\frac{R^{(\infty)}(t)-\frac{1}{t}}{t}.
\end{align*}
Of course, $L_1^{(\infty)}=0$ uniformly here. 
Recall that by using Theorem \ref{BryantSolitonsmalltime} and Propositon \ref{initialestaimtesxy}, we find that the corresponding 
$x^{(\infty)}$ and $y^{(\infty)}$ functions satisfy \eqref{initialyestimatebryant}.
It is also well-known that $x^{(\infty)}$ and $y^{(\infty)}$ are monotonically decreasing functions. 

We prove this lemma by comparing $(\tilde{\xi},\tilde{L_1},\tilde{L_2},\tilde{R})$ to $(\xi^{(\infty)},L_1^{(\infty)},L_2^{(\infty)},R^{(\infty)})$ on 
the interval $[0,\frac{1}{9}]$. We use the notation $\tilde{\eta}$ and $\eta^{(\infty)}$ to mean the functions from $[0,\infty)$ to $\mathbb{R}^4$ 
consisting of components formed by breaking the corresponding sets of functions according to \eqref{etaforma}. Letting 
$\textbf{v}=\tilde{\eta}-\eta^{(\infty)}$, we find 
\begin{align}\label{vsystem}
 \textbf{v}'(t)=\frac{A\textbf{v}}{t}+B(\textbf{v})-P, \qquad \textbf{v}(0)=0, \qquad \textbf{v}'(0)=(-p^2,-\frac{p^2}{3},0,0), 
\end{align}
where 
\begin{align*}
 A= \begin{pmatrix}
     0&0&-4&0\\
     0&-2&0&0\\
     -1&0&-2&2\\
     0&0&-1&-1
    \end{pmatrix} \qquad 
    P=\begin{pmatrix}
       p^2\\
       p^2\\
       p^2\\
       0
      \end{pmatrix}
\end{align*}
and 
\begin{align*}
 \left|B(\textbf{v})\right|_{\infty}\le 5\left|\textbf{v}\right|_{\infty},
\end{align*}
provided $\left|\tilde{\eta}\right|_{\infty}\le 1$ and $\left|\textbf{v}\right|_{\infty}\le \frac{1}{100}$. One can easily check 
using the results in the Appendix that $\left|\tilde{\eta}\right|_{\infty}\le 1$ on $[0,t^*]$. 

Now we let $\textbf{s}=S^{-1}\textbf{v}$ so that 
\begin{align*}
 \textbf{s}'=\frac{S^{-1}AS\textbf{s}}{t}+S^{-1}B(S\textbf{s})-S^{-1}P, \qquad  \textbf{s}(0)=0, \qquad \textbf{s}'(0)=S^{-1}\textbf{v}'(0),
\end{align*}
where 
\begin{align*}
 S=\begin{pmatrix}
    2&-1&0&8\\
    0&0&1&0\\
    1&-1&0&-2\\
    1&0&0&1
   \end{pmatrix},
\qquad 
 S^{-1}=\begin{pmatrix}
   -\frac{1}{9}&0&\frac{1}{9}&\frac{10}{9}\\
   -\frac{1}{3}&0&-\frac{2}{3}&\frac{4}{3}\\
   0&1&0&0\\
   \frac{1}{9}&0&-\frac{1}{9}&-\frac{1}{9}
   \end{pmatrix}, \qquad S^{-1}AS=\begin{pmatrix}
   -2&1&0&0\\
   0&-2&0&0\\
   0&0&-2&0\\
   0&0&0&1
   \end{pmatrix}.
\end{align*}
This almost-diagonal form of the equations makes it clear that 
\begin{align*}
 \left|\textbf{v}(t)\right|_{\infty}\le 11\left|\textbf{s}(t)\right|_{\infty}\le 10^6p^2t\le \frac{1}{100}
\end{align*}
for all $t\in [0,t^*]$, since $p$ is small. 

The rest of the proof involves using the smallness of $\textbf{v}$ to obtain the estimates discussed in the statement of the lemma. 
First note that 
\begin{align*}
 \left|x(pt)-x^{\infty}(t)\right|&\le \left|\frac{\tilde{L_2}(t)}{\tilde{R}(t)}-\frac{L_2^{(\infty)}(t)}{R^{(\infty)}(t)}\right|\\
 &\le t^2 \left|\left(R^{(\infty)}(t)\tilde{L_2}(t)-L_2^{(\infty)}(t)\tilde{R}(t)\right)\right|\\
 &\le t^2 \left(R^{(\infty)}(t)\left|\tilde{L_2}(t)-L_2^{(\infty)}(t)\right|
 +\left|L_2^{(\infty)}(t)\right|\left|R^{(\infty)}(t)-\tilde{R}(t)\right|\right)\\
 &\le 2\cdot 10^6 p^2 t^2, \qquad t\in [0,\frac{1}{9}]. 
\end{align*}
In this previous computation, we used the estimates on $R^{(\infty)}$ and $L_2^{(\infty)}$ from Theorem \ref{BryantSolitonsmalltime} in the Appendix. 
Similarly, 
\begin{align*}
 \left|y(pt)-y^{\infty}(t)\right|&\le \left|\frac{\tilde{R}(t)}{\tilde{\xi}(t)}-\frac{R^{(\infty)}(t)}{\xi^{(\infty)}(t)}\right|\\
 &\le \frac{1}{\tilde{\xi}(t)\xi^{(\infty)}(t)}\left|\left(R^{(\infty)}(t)\tilde{\xi}(t)-\xi^{(\infty)}(t)\tilde{R}(t)\right)\right|\\
 &\le t^2\left(R^{(\infty)}(t)\left|\tilde{\xi}(t)-\xi^{(\infty)}(t)\right|
 +\xi^{(\infty)}(t)\left|R^{(\infty)}(t)-\tilde{R}(t)\right|\right)\\
 &\le 4\cdot 10^6p^2 t^2, \qquad t\in [0,\frac{1}{9}].
\end{align*}
To check the closeness of $w$ and $z$ to $0$, first note that 
\begin{align}\label{basiczestimate}
0\le z(pt)\le pt
\end{align}
since $\left|R'\right|\le R^2$ by Proposition \ref{sf2p}. Also, the above estimates for $\textbf{v}$ and $\eta^{(\infty)}$ 
imply that $\tilde{\xi}(\frac{1}{9})\ge 0$, so the equation for $\tilde{L}_1$ implies immediately that 
$0\le -\tilde{L}_1'(t)\le p^2$ for each $t\in [0,t^*]$, so we find
\begin{align*}
 0\le -w(pt)=-L_1(pt)=\frac{-\tilde{L}_1(t)}{p}\le pt, \qquad t\in [0,\frac{1}{9}].
\end{align*}

To conclude the proof, we need show that $\min\{x,E\}\ge 0$ on this interval. It is clear that $x\ge 0$. To show that 
$E\ge 0$, it
suffices to show that $\left|z(pt)\right|^2\le \frac{\left|1-x(pt)\right|}{10}$ for all $t\in [0,\frac{1}{9}]$ because of the inequality $0\le-w\le 1$, 
the equation for $E'$ in \eqref{usefulquantities} and the fact that $\lim_{t\to 0}\xi L_2+1-R^2-L_2^2=\eta_0'(0)+1-2\eta_3'(0)=-3\eta_2'(0)=6\delta_1$, 
so that $\xi L_2+1-R^2-L_2^2$ is initially positive. 
From \eqref{basiczestimate}, we have
\begin{align*}
 0\le \frac{z(pt)}{pt}\le 1,
\end{align*}
but we also have 
\begin{align*}
\frac{1-x(pt)}{p^2t^2}&=\frac{1-\tilde{x}(t)}{p^2 t^2}\\
&=\frac{1-x^{(\infty)}(t)}{p^2 t^2}+\frac{x^{(\infty)}(t)-\tilde{x}(pt)}{p^2 t^2}\\
&\ge \frac{2 \tan^2\left(\sqrt{\frac{3}{2}}t\right)}{t^2 p^2}-4\cdot 10^6,
\end{align*}
by Theorem \ref{BryantSolitonsmalltime},
so we obtain the required estimates. 
\end{proof}

\section{Bounds on curvature at the $\mathbb{S}^2$ singular orbit}
By Theorems \ref{S2orbit} and \ref{d1large}, we find that any Ricci soliton must satisfy $\delta_1\in [0,10^{20,000}]$. We now construct bounds on
$\delta_2$ and $\delta_3$. Fortunately, these bounds are simpler to construct, 
and can be found \textit{without} using the (already large) bound on $\delta_1$. The bound on $\delta_2$ is easier to construct once we have a bound on $\delta_3$, 
so we treat $\delta_3$ first. 
\begin{thm}\label{delta3estimate}
 Suppose the metric of the form \eqref{metricform} is a gradient shrinking Ricci soliton with Einstein constant $\lambda=1$. Then $\delta_3\in [0,40]$. 
\end{thm}
\begin{proof}
We already know from Proposition \ref{initialbounds} that $\delta_3\ge 0$, so suppose for the sake of contradiction that $\delta_3>40$. 
In this case, we claim that 
\begin{align}\label{Tbiggerthan5}
 T>\frac{1}{2} \ \text{and} \ L_2(t)<0 \ \text{for all} \ t\in (T-\frac{1}{2},T).
\end{align}
To verify \eqref{Tbiggerthan5}, note that if $T\le \frac{1}{2}$, then Proposition \ref{sf2p} implies that 
$\left|R'\right|\le R^2$, so that $R>1$ on $(0,T)$. Then the estimate $L_2'=-\xi L_2+R^2-1>-\xi L_2$ violates the boundary conditions 
$\lim_{t\to 0}L_2(t)=\lim_{t\to 0}\xi(t)=+\infty$ and $\lim_{t\to T}L_2(t)=0$, $\lim_{t\to T}\xi(t)=-\infty$. Since $T>\frac{1}{2}$, Proposition \ref{sf2p}
 again implies that 
$R^2-1>0$ on $(T-\frac{1}{2},T)$, so the boundary condtions $L_2(T)=0$, $\lim_{t\to T}\xi(t)=-\infty$ and the inequality $L_2'>-\xi L_2$ together imply 
that $L_2<0$. 

With \eqref{Tbiggerthan5} in hand, we now claim that
\begin{align}\label{l2l1estimates}
 L_2(t)\ge L_1(t)\ge \frac{1}{t-T} \ \text{for all} \ t\in (0,T). 
\end{align}
The second inequality in \eqref{l2l1estimates} is an immediate consequence of the fact that $L_1'\le -L_1^2$ on $(0,T)$ (follows from Proposition \ref{scf1}), 
and $\lim_{t\to T}L_1(t)=-\infty$. 
The first inequality is a consequence of the fact that 
 \begin{align*}
  (L_2-L_1)'&=-\xi(L_2-L_1)+R^2,\\
 \end{align*}
and the observation that $\lim_{t\to 0}(L_2-L_1)(t)=+\infty$. 
In fact, we know that $\xi L_1+1-L_1^2\ge 0$ everywhere (Proposition \ref{scf1} again), so since $L_1\le 0$ everywhere, 
we rearrange to find $-\xi \ge \frac{1}{L_1}-L_1$ so we can estimate further on $(T-\frac{1}{2},T)$:
\begin{align*}
 (L_2-L_1)'&\ge (\frac{1}{L_1}-L_1)(L_2-L_1)+R^2\\
 &=-L_1(L_2-L_1)+\frac{L_2-L_1}{L_1}+R^2\\
 &\ge (L_2-L_1)^2+\frac{L_2-L_1}{L_1}+R^2\\
 &\ge (L_2-L_1)^2+R^2-1\\
 &\ge (L_2-L_1)^2+\frac{1}{(\frac{1}{\delta_3}+T-t)^2}-1
\end{align*}
since $L_2<0$ (follows from \eqref{Tbiggerthan5}) and $R'\le R^2$ (Proposition \ref{sf2p}). Let $y=(L_2-L_1)(T-t)$, so that $y(T-\frac{1}{2})\ge 0$ and $y(t)\le 1$ for all $t\in (T-\frac{1}{2},T)$. But 
we can estimate the evolution of $y$: 
\begin{align*}
 y'&\ge \left((L_2-L_1)^2+\frac{1}{(\frac{1}{\delta_3}+T-t)^2}-1\right)(T-t)-(L_2-L_1)\\
 &= \frac{y^2-y}{(T-t)}+\left(\frac{1}{(\frac{1}{\delta_3}+T-t)^2}-1\right)(T-t)\\
 &\ge -\frac{1}{4(T-t)}+\frac{(T-t)}{(\frac{1}{\delta_3}+T-t)^2}-(T-t),
\end{align*}
since $y^2-y\ge -\frac{1}{4}$ for all $y\in \mathbb{R}$. 
If $\delta_3>40$, we integrate to get the following estimate: 
\begin{align*}
 1&\ge y(T-\frac{1}{40})-y(T-\frac{1}{2})\\
 &=\int_{T-\frac{1}{2}}^{T-\frac{1}{40}}y'(t)dt\\
 &\ge \int_{T-\frac{1}{2}}^{T-\frac{1}{40}} \left(\frac{(T-t)}{(\frac{1}{\delta_3}+T-t)^2}-(T-t)-\frac{1}{4(T-t)}\right)dt\\
 &\ge 1.89 -\frac{1}{8}-\frac{\ln(20)}{4}>1,
\end{align*}
which is a contradiction. 
\end{proof}

\begin{thm}
 Suppose the metric of the form \eqref{metricform} is a gradient shrinking Ricci soliton with Einstein constant $\lambda=1$. Then $\delta_2\in [-1,10^{20,000}]$. 
\end{thm}
\begin{proof}
Once again consider the non-negative quantity defined by $K(t)^2=L_2(t)^2+(R(t)-1)^2$, and note that 
\begin{align}\label{Kestimateagain}
 K'(t)\ge K(t)\left(-\frac{1}{2}-\max\{0,\xi\}\right).
\end{align}
Also note that 
\begin{align}\label{maxtimetillT}
 T-\xi^{-1}(0)\le 2
\end{align}
because of the estimate $y'\ge y^2+1$, where $y=\min\{-\xi,-L_1\}$ and $y(\xi^{-1}(0))\ge 0$. 
Now Theorem \ref{delta3estimate} tells us that $\lim_{t\to T}K(t)\le 39$, so \eqref{Kestimateagain} combined with \eqref{maxtimetillT} 
implies that 
\begin{align}\label{d2kestimate}
 K(t)\le 39e \ \text{for all} \ t\in (\xi^{-1}(0),T). 
\end{align}
The equation $L_2'=-\xi L_2+R^2-1$ then implies that $\left|L_2(t)\right|\le e(39^2e+2\cdot 39)(T-t)$ for all $t\in 
(\xi^{-1}(0),T)$. Now consider the quantities $X=\xi-L_1$ and $Y=\xi L_1+1-L_1^2$. 
Since $Y\ge 0$ everywhere, we find that $L_1(\xi^{-1}(0))\in [-1,0]$ and $X(\xi^{-1}(0))\in [0,1]$. 
We compute 
\begin{align}\label{xestimated2}
 X'=L_1X-2L_2^2, 
\end{align}
and 
\begin{align*}
 Y'&=L_1'(\xi-L_1)+L_1(\xi'-L_1')\\
 &=(-\xi L_1-1)(\xi-L_1)+L_1(-L_1^2-2L_2^2-1+\xi L_1+1)\\
 &=(-\xi L_1-1+L_1^2)(\xi-L_1)-2L_1L_2^2\\
 &=-XY-2L_1L_2^2.
\end{align*}
Since $L_1\le 0$ and $T-\xi^{-1}(0)\le 2$, we find from \eqref{d2kestimate} and \eqref{xestimated2} that $-X(t)\le 4\cdot (39e)^2$ for all 
$t\in (\xi^{-1}(0),T)$. 
Therefore 
\begin{align*}
 Y'\le 4\cdot (39e)^2 Y+2(39^2e^2+2\cdot 39e)^2(T-t),
\end{align*}
so that $Y(T)$, which coincides with $\frac{3}{2}(\delta_2+1)$, can be no more than $10^{20,000}$.

\end{proof}

\section{Compactness and uniqueness}\label{mainproof}

We summarise the results for the $\mathbb{S}^4$ we have seen so far. 
\begin{thm}\label{deltaboundedsummary}
An $SO(2)\times SO(3)$-invariant gradient shrinking Ricci soliton on $\mathbb{S}^4$ of the form \eqref{metricform} has $\delta_1\in [0,10^{20,000}]$, $\delta_2\in [-1,10^{20,000}]$ 
and $\delta_3\in [0,40]$. 
\end{thm}
We are now ready to prove the main result of this paper. 
\begin{proof}[Proof of Theorem \ref{CT}]
We assume for the sake of contradiction that there is no such value of $\mathcal{C}>0$. 
Then there is a sequence of $SO(2)\times SO(3)$-invariant solutions $(g,u)$ to \eqref{GRS} with unbounded Riemann curvature, 
unbounded volume, or an injectivity radius shrinking to $0$. 
Theorem \ref{deltaboundedsummary} implies that 
$\delta_1,\delta_2,\delta_3$ are all bounded uniformly, so we can assume that there numbers are all convergent. Propositions \ref{mapa} and \ref{mapb} 
imply that our sequence of solutions converge to another solution. It is clear that the Riemann curvature, volume and injectivity radii all depend continuously on 
the values of $(\delta_1,\delta_2,\delta_3)$, so we obtain a contradiction. 
\end{proof}

With Theorem \ref{CT} in hand, we discuss how one could prove Conjecture \ref{ClT}. First note that by Propositions \ref{mapa} and \ref{mapb}, there 
is a smooth function 
$F:\mathbb{R}^3\to \mathbb{R}^3$ whose zeroes are precisely the Ricci solitons we aim to classify. 
By Theorem \ref{deltaboundedsummary}, there is a compact domain
$\Omega\subset \mathbb{R}^3$ that contains all the zeroes of $F$. 
In fact, we have \textit{explicitly} described this domain. Therefore, Conjecture \ref{CT} would follow with the following steps. 

\begin{enumerate}
\item Find an explicit open neighbourhood $N$ of the canonical metric on $\mathbb{S}^4$ on which no other zeroes of $F$ occur.
This is essentially a quantitative use of the inverse function theorem. 
\item Show numerically that there are no zeroes of $F$ in $\Omega\setminus N$. This could be achieved by finding an upper bound for $\left|dF\right|$ on $\Omega$, discretising the set 
$\Omega\setminus N$ accordingly, and showing that $F$ is sufficiently far away from $0$ at each of these finitely-many points using an appropriate numerical ODE solver. 
\end{enumerate}
We do not pursue these ideas in this paper, primarily because the $\delta_1,\delta_2$ bounds we have found are far too large for the numerics 
described here to provide an answer in a reasonable amount of time. However, we do emphasise that these techniques 
could be used to resolve Conjecture \ref{ClT} in the affirmative in `finite time'.

\section{An $SO(2)\times SO(3)$-invariant ancient solution on $\mathbb{S}^4$}\label{newancient}
The sequence of `almost Ricci solitons' on $\mathbb{S}^4$ we found by making $\delta_1$ large appears to have a pancake shape. In this section, we describe 
a `pancake' $\kappa$-noncollapsed ancient solution to the Ricci flow; it is likely that this is the precise geometric structure that our 
`almost Ricci solitons' are detecting. It is worth noting that by the recent classification result in \cite{Brendle}, 
this $\kappa$-noncollapsed ancient Ricci flow on $\mathbb{S}^4$ cannot be uniformly PIC and 
weakly PIC2. 

We restate 
Theorem \ref{NAS} for convenience. 
\begin{thm}\label{ancientsolution}
 There exists a $\kappa>0$ and a $\kappa$-noncollapsed $SO(2)\times SO(3)$-invariant ancient Ricci flow on $\mathbb{S}^4$ with positive Riemann curvature operator 
 which is not isometric to the round 
 shrinking sphere. The group $SO(2)\times SO(3)$ acts on $\mathbb{S}^4\subseteq \mathbb{R}^5=\mathbb{R}^2\oplus \mathbb{R}^3$ in the obvious way. 
\end{thm}
The proof of this result is broken up into several steps. Apart from the first step, 
the construction of this ancient Ricci flow solution is almost identical to that of the Perelman ancient `sausage' solution; 
the details are available in Chapter 19 of \cite{RicciIII}. 
\begin{proof}
\textbf{Step One: a sequence of initial Riemannian metrics.}
For each large $L>10$, choose an $SO(2)\times SO(3)$-invariant Riemannian metric $g_0$ on $\mathbb{S}^4$ of the form \eqref{metricform} with $T=L+1$ and 
\begin{align*}
 \tilde{f}_1(r)=\begin{cases}
         L \ &\text{if}  \ 0<r<1\\
         (L+1-r) \ &\text{if}  \ 1\le r<L+1
        \end{cases}, \qquad 
        \tilde{f}_2(r)=\begin{cases}
        \sin(\frac{\pi r}{2}) \ &\text{if} \ 0<r<1\\
        1 \ &\text{if} \ 1\le r<L+1. 
       \end{cases}
\end{align*}
This metric clearly satisfies the smoothness conditions at the singular orbits, but it does fail to be smooth at $r=1$. However, we can mollify the two functions 
$\tilde{f}_1,\tilde{f}_2$ on $(-\frac{1}{2},\frac{3}{2})$ so that the resulting functions $f_1,f_2$ are smooth. 
Recall that the Riemann curvatures of this Riemannian metric are 
\begin{align*}
 -\frac{f_1''}{f_1}, \qquad -\frac{f_2''}{f_2}, \qquad \frac{1-(f_2')^2}{f_2^2}, \qquad \frac{-f_1'f_2'}{f_1f_2};
\end{align*}
it is clear that after mollification, these curvatures are all non-negative. 
Since we mollify on $(\frac{1}{2},\frac{3}{2})$, the quantity $-\frac{f_2''}{f_2}+\frac{1-(f_2')^2}{f_2^2}$ is uniformly bounded from below, independently of 
$L$ because $f_2$ does not depend on $L$ in this region. Also the supremum of all four eigenvalues is uniformly bounded from above ($L$ does not affect the $f_2$ terms, and only 
makes $f_1$ large so the corresponding curvatures can only get smaller). Therefore, 
\begin{align}\label{initialcest}
 \textit{there exists a} \  C>0 \  \textit{so that} \ \frac{1}{C}\le S(g_0)=\left|\mathcal{R}(g_0)\right|\le C \ \textit{for all} \ L>10. 
\end{align}
Now we claim that there is a $\kappa_0>0$ so that $(M,g_0)$ is $\kappa_0$-noncollapsed on all scales $r_0>0$, and for all large $L>10$. To see this, we must show that any geodesic ball $B_{r_0}(x_0)$ on which $S(g_0)\le \frac{1}{r_0^2}$ has volume at least $\kappa_0 r_0^4$. By the lower bound on the scalar curvature \eqref{initialcest}, it suffices to find a $\kappa_0>0$ so that the volume of any geodesic ball $B_{r_0}(x_0)$ is at least $\kappa_0 r_0^4$ whenever $0<r_0^2\le C$.  
To this end, 
take an arbitrary geodesic ball $B_{r_0}(x_0)\subseteq \mathbb{S}^4$ and use the manifold decomposition:
 \begin{align*}
  \mathbb{S}^4=\overline{A\cup B},
 \end{align*}
where $A=(0,\frac{3}{2}]\times \mathbb{S}^1\times \mathbb{S}^2$ and $B=[\frac{3}{2},L+1)\times \mathbb{S}^1\times \mathbb{S}^2$. 
Now for each $L>10$, $B$ is isometrically contained in $\mathbb{R}^2\times \mathbb{S}^2$ equipped 
 with the standard metric, so we have 
\begin{align}\label{gaussiankappa}
\inf \Bigg\{ \  \frac{\text{Vol}(B_{r_0}(x_0))}{r_0^4} \  \vert\  B_{r_0}(x_0)\subseteq B, r_0\in (0,C]\ \Bigg\}\ge \inf \Bigg\{ \  \frac{\text{Vol}(B_{r_0}(x_0))}{r_0^4} \  \vert\  B_{r_0}(x_0)\subseteq \mathbb{R}^2\times \mathbb{S}^2, r_0\in (0,C] \ \Bigg\}.
\end{align}
On the other hand, 
let $\overline{g}$ be an $SO(3)$-invariant Riemannian metric on $\mathbb{S}^3$ found by extending $f_2$ on $(0,\frac{3}{2})$ 
smoothly to a function on $(0,2)$ with the appropriate smoothness conditions at $2$. Also let $\overline{f_1}(r)$ be a 
smooth extension of $f_1:(0,\frac{3}{2})\to \mathbb{R}$ to a function with domain $(0,2)$ and appropriate smoothness conditions at $r=2$. 
Note that $\overline{g}$ is independent of $L$, 
and $A$ is isometrically embedded in the warped product manifold $\mathbb{S}^1\times \mathbb{S}^3$ with metric $(\overline{f_1}(r)^2d\theta^2,\overline{g})$, where $d\theta$ 
 is the standard one-form on $\mathbb{S}^1$. 
 Therefore, 
\begin{align}\label{warpedkappa}
\inf \Bigg\{ \  \frac{\text{Vol}(B_{r_0}(x_0))}{r_0^4} \  \vert\  B_{r_0}(x_0)\subseteq A, r_0\in (0,C] \ \Bigg\}\ge \inf \Bigg\{ \  \frac{\text{Vol}(B_{r_0}(x_0))}{r_0^4} \  \vert\  B_{r_0}(x_0)\subseteq \mathbb{S}^1\times \mathbb{S}^3, r_0\in (0,C] \ \Bigg\}.
\end{align} 
 Now it is well-known that the right hand side of \eqref{gaussiankappa}, which is independent of $L$, is strictly positive. Also, we can arrange it so that $\overline{f_1}(r)\ge 1$ for all $r\in (0,2)$ and $L>10$, so it is also clear that the right hand side of \eqref{warpedkappa} is strictly positive for all $L>10$, and can be bounded from below by a positive number, independently of $L>10$. 
 
Now any other geodesic ball of radius $r_0$ contains a geodesic 
 ball of radius $\frac{r_0}{2}$ which is entirely contained in at least one of $A$ or $B$, so the existence of a $\kappa_0>0$ independent of $L$ follows from \eqref{gaussiankappa} and \eqref{warpedkappa}. 
 
\textbf{Step Two: Ricci flows from the initial metrics.}
Since the Riemann curvature of $g_0$ is non-negative, for each $L>10$ there exists a 
$T_L$ so that the Ricci flow starting at $g_0$ becomes singular for the first time at $T_L$, and with the round sphere 
as a singularity model. By \eqref{initialcest}, there exists a uniform $C'>0$ so that $T_L\in (\frac{1}{C'},C')$ for all $L>10$. Furthermore, by \eqref{initialcest} 
and the fact that each $g_0$ is $\kappa_0$-noncollapsed on all scales, Theorem 19.52 (no local collapsing) of \cite{RicciIII} implies the existence of a $\kappa>0$ 
so that all of the Ricci flows on $(\frac{2T_L}{3},T_L)$ are uniformly $\kappa$-noncollapsed on all scales $r^2<\frac{T_L}{3}$. 
Note that we can apply this theorem because we can uniformly bound $\left|\mathcal{R}\right|$ on the time interval $[0,\frac{1}{16\mathcal{C}}]$ 
independently of $L$; this is a result of the evolution equations for $\mathcal{R}$ discussed in the proof of Proposition \ref{scf1}.

For a given small $\delta>0$, let $t_L$ be the unique time at which
the ratio of largest to smallest Riemann curvature eigenvalues is 
$1+\delta$, and so that the ratio is less than or 
equal to $1+\delta$ on $(t_L,T_L)$. 

\textbf{Step Three: rescaled flows.}
Consider the rescaled Ricci flows 
\begin{align*}
 \hat{g}_L(t)=\frac{1}{T_L-t_L}g_L(T_L+(T_L-t_L)t)
\end{align*}
which start at time $t_0(L)=-\frac{T_L}{T_L-t_l}$, have sectional curvature ratio $1+\delta$ at time $-1$, become singular at time $0$, 
and are uniformly $\kappa$-noncollapsed on $(-\frac{T_L}{3(T_L-t_L)},0)$ for scales $r^2\le \frac{-t_0(L)}{3}$. 

We now show that $\lim_{L\to \infty}t_0(L)=-\infty$. To see this, note that for small enough $\delta>0$, we have 
\begin{align}\label{scalacurvaturebound-1}
 \hat{S}(-1)\le 3,
\end{align}
thanks to the evolution equation 
\begin{align*}
 \frac{\partial \hat{S}}{\partial t}\ge \Delta_{\hat{g}(t)}\hat{S}+\frac{\hat{S}^2}{2}
\end{align*}
and the pinching of $1+\delta$ that we have already established. 
The estimate \eqref{scalacurvaturebound-1} coupled with 
Hamilton's trace Harnack inequality (Corollary 15.3 in \cite{RicciII}) then 
implies that $\hat{S}(t)\le 3\frac{-1-t_0(L)}{t-t_0(L)}$ for $t\in (t_0(L),-1]$. 
Therefore, 
\begin{align*}
 \frac{L}{\sqrt{T_L-t_l}}&\le \text{diam} (\hat{g}_L(t_0(L)))\\
 &= \text{diam} (\hat{g}_L(-1))-\int_{t_0(L)}^{-1}\frac{\partial \text{diam}(\hat{g}_L(t))}{\partial t}dt\\
 &\le \text{diam} (\hat{g}_L(-1))+12\int_{t_0(L)}^{-1}\sqrt{\frac{-1-t_0(L)}{t-t_0(L)}}dt\\
 &=\text{diam} (\hat{g}_L(-1))+24(-1-t_0(L))\\
 &\le 24\frac{T_L}{T_L-t_l}.
\end{align*}
Note that in the last computation, we have used the following facts:
\begin{itemize}
 \item For any $x,y\in \mathbb{S}^4$, $\frac{\partial d(x,y)}{\partial t}\ge -12 \sqrt{\frac{-1-t_0(L)}{t-t_0(L)}}$, which follows from the Ricci curvature upper bound.
 \item By making $\delta$ small, we can force $\hat{g}_L(-1)$ to be arbitrarily close to a round sphere, and the scalar curvature of this particular sphere must be $2$ 
 because the remaining time until blow up is exactly $1$. Therefore, Myer's theorem implies that 
 $\text{diam} (\hat{g}_L(-1))\le 2\sqrt{3}\pi$. 
\end{itemize}
The estimate $\frac{L}{\sqrt{T_L-t_l}}\le 24\frac{T_L}{T_L-t_l}$ implies that $\lim_{L\to \infty}T_L-t_L= 0$ as well, so 
$\lim_{L\to \infty}t_0(L)=-\infty$, because of our uniform estimates on $T_L$ itself. 

\textbf{Step Four: convergence.}
We are now in a position to take limits. Indeed, the uniform $\kappa$-noncollapsing and curvature bounds are enough to get smooth convergence to an ancient 
Ricci flow by Theorem 3.10 in \cite{RicciI}; the ancient Ricci flow must have non-negative (hence positive) Riemann curvature. Since $\delta>0$
was eventually fixed, we can examine the sequence at time $-1$ to find that the ancient Ricci flow solution is on $\mathbb{S}^4$, but it is not the round sphere. 
By following the arguments in Chapter 19 of \cite{RicciI}, we find that 
$SO(2)\times SO(3)$ acts via isometries on this ancient solution on $\mathbb{S}^4\subset \mathbb{R}^5=\mathbb{R}^2\oplus \mathbb{R}^3$ in the obvious way. 
\end{proof}

We conclude this section by observing that, not only is this $SO(2)\times SO(3)$-invariant ancient solution non-round, 
but it is not isometric to Perelman's rotationally-invariant $\kappa$ solution either. 
\begin{prop}\label{so432}
 Let $SO(2)\times SO(3)$ act on $\mathbb{S}^4\subset \mathbb{R}^5=\mathbb{R}^2\oplus \mathbb{R}^3$ in the obvious way. Any $SO(2)\times SO(3)$-invariant continuous 
 function 
 $f:\mathbb{S}^4\to \mathbb{R}$ which is also invariant under an $SO(4)$ rotation group action must be constant. 
\end{prop}
\begin{proof}
Let $r_1:\mathbb{S}^4\to [0,1]$ be a continuous function which parametrises the $SO(4)$ action, and let $r_2:\mathbb{S}^4\to [0,1]$ be a continuous function which 
parametrises the $SO(2)\times SO(3)$ action. Note that, almost by construction, we have the following:
\begin{itemize}
 \item $r_1^{-1}(0)$ and $r_1^{-1}(1)$ are single points;
 \item $r_2^{-1}(0)$ is a copy of $\mathbb{S}^1$;
 \item $r_2^{-1}(1)$ is a copy of $\mathbb{S}^2$;
 \item $r_1^{-1}(k)$ is a copy of $\mathbb{S}^3$ for each $k\in (0,1)$;
 \item $r_2^{-1}(k)$ is a copy of $\mathbb{S}^1\times \mathbb{S}^2$ for each $k\in (0,1)$;
 \item the function $f$ is constant on the level sets $r_i^{-1}(k)$ for each $i=1,2$ and $k\in [0,1]$. 
\end{itemize}
We claim that there is an $\epsilon>0$ so that $f$ is constant on $r_1^{-1}([0,\epsilon])$ and $r_1^{-1}([\epsilon,1])$. To see this, note that $f$ must be constant on 
the two submanifolds $(SO(2)\times SO(3))\cdot r_1^{-1}(0)$ and $(SO(2)\times SO(3))\cdot r_1^{-1}(1)$. Since these submanifolds are compact, connected 
and at least one-dimensional, their images under 
the continuous function $r_1$ must be non-trivial closed subintervals of $[0,1]$ containing $0$ or $1$, respectively.  
However, since $f$ is constant on the level sets of $r_1$, $f$ must actually be constant on the $r_1$ pre-images of these closed subintervals. 

Choose $\tilde{\epsilon}_0\in (0,1]$ so that $[0,\tilde{\epsilon}_0]$ is the maximal connected subinterval of $[0,1]$ containing $0$ so that 
$f$ is constant on $[0,\tilde{\epsilon}_0]$. Similarly, choose $\tilde{\epsilon}_1\in [0,1)$ so that $[\epsilon_1,1]$ is the maximal connected subinterval of $[0,1]$ 
containing $1$ so that $f$ is constant on $[\epsilon_1,1]$. The argument in the last paragraph shows that $\tilde{\epsilon}_0$ and $\tilde{\epsilon}_1$ both exist. 
If $\tilde{\epsilon}_1\le \tilde{\epsilon}_0$, the proof will be complete, so we assume that $0<\tilde{\epsilon}_0<\tilde{\epsilon}_1<1$. 
Define the pairwise-disjoint connected sets $A=r_1^{-1}([0,\tilde{\epsilon_0}])$, 
$B=r_1^{-1}((\tilde{\epsilon}_0,\tilde{\epsilon}_1))$, $C=r_1^{-1}([\tilde{\epsilon}_1,1])$. 
It is clear that $\mathbb{S}^4=A\cup B\cup C$. 
Also note that both $A$ and $C$ are $SO(2)\times SO(3)$-invariant because any orbit containing 
points both in and out of either of the sets would violate the maximality of $[0,\tilde{\epsilon_0}]$ or $[\tilde{\epsilon}_1,1]$. This implies that 
$B$ is $SO(2)\times SO(3)$-invariant as well. 
Now $r_2(A),r_2(B)$ and $r_2(C)$ are connected subintervals which must cover $[0,1]$; they must be pairwise disjoint 
because $SO(2)\times SO(3)$ orbits stay in exactly one of $A$, $B$ or $C$. 
Since $A$ and $C$ are compact, we find that $r_2(A)$ and $r_2(C)$ are closed, so 
$r_2(B)$ must be of the form $(a,c)$ for some $0<a<c<1$. Therefore, $B$ must be homeomorphic to both $(a,c)\times \mathbb{S}^1\times \mathbb{S}^2$ and 
$(\tilde{\epsilon}_0,\tilde{\epsilon}_1)\times \mathbb{S}^3$, which is a contradiction (these two topological spaces have non-isomorphic fundamental groups).

\end{proof}

\appendix

\section{The Bryant soliton revisited}\label{BryantAppendix}
The Bryant steady soliton is a rotationally-invariant metric on $\mathbb{R}^3$ of the form $g=dt\otimes dt+f(t)^2 Q$, where 
$Q$ is the standard metric on $\mathbb{S}^2$ with Ricci curvature $1$, and $f:(0,\infty)\to (0,\infty)$ is smooth, and can be extended 
to a smooth and odd function on $(-\infty,\infty)$ with $f'(0)=1$. A detailed construction of this metric is given in \cite{RicciI}, 
where the authors also show that the Bryant soliton has 
positive Riemann curvature everywhere. 
If we let $R=\frac{1}{f}$, $L_2=\frac{f'}{f}$ and $\xi=2L_2-u'$, where $u$ is the potential function, then 
\begin{align}
\begin{split}\label{Bryantsolutionnonrescaled}
 \xi'&=-2L_2^2,\\
 L_2'&=-\xi L_2+R^2,\\
 R'&=-L_2 R.
 \end{split}
\end{align}
For any $p>0$, solutions of \eqref{Bryantsolutionnonrescaled} are invariant under the transformation that sends a function $f(t)$ to $pf(pt)$, so 
to uniquely specify the Bryant soliton, we need to prescribe the value $\delta_1:=\lim_{t\to 0}\left(\frac{R(t)-\frac{1}{t}}{t}\right).$

Let $x=\frac{L_2}{R}$, $y=\frac{R}{\xi}$ and $z=\frac{1}{R}$, then 
\begin{align}\label{bryantfirstchange}
\begin{split}
 x'&=\frac{1}{yz}\left(-x+y+yx^2\right),\\
 y'&=\frac{-xy^2+2x^2y^3}{yz},\\
 z'&=x.
 \end{split}
\end{align}
The following facts about the Bryant soliton curve $(x(t),y(t),z(t))$ are well-known (they are also discussed in \cite{RicciI}):
\begin{itemize}
 \item $(x(0),y(0))=(1,\frac{1}{2})$;
 \item  $\lim_{t\to \infty}(x(t),y(t))=(0,0)$;
 \item $x'(t),y'(t)< 0$ for all $t\in (0,\infty)$;
 \item $\lim_{t\to \infty}\frac{y(t)}{x(t)}=1$. 
\end{itemize}
It is therefore clear that there is a function $f:[0,1]\to [0,\frac{1}{2}]$ so that $y(t)=f(x(t))$ along the Bryant soliton, and this function does not 
depend on the choice of $\delta_1>0$ (this parameter only affects how quickly one travels through the curve $y=f(x)$). The following propositions tell us 
some valuable information about this function. 
\begin{prop}\label{initialestaimtesxy}
 The function $f$ satisfies $\frac{x}{2}\le f(x)\le \frac{x}{2}+(1-x)^2$ for all $x\in [\frac{3}{4},1]$. In fact, $\frac{x}{2}\le f(x)$ for all $x\in [0,1]$. 
\end{prop}
\begin{proof}
Consider the system 
\begin{align}\label{xysystmebryant}
\begin{split}
 x'&=-x+y+yx^2,\\
 y'&=-xy^2+2x^2y^3;
 \end{split}
\end{align}
the critical point $(1,\frac{1}{2})$ has a one-dimensional unstable manifold. The part of the unstable manifold lying in $[0,1]\times [0,\frac{1}{2}]$ consists 
of exactly those points of the form $(x,f(x))$ for $x\in [0,1]$. It therefore suffices to show that any point in the unstable manifold with 
$x\in [\frac{3}{4},1]$ satisfies
\begin{align}\label{bryantyxestimatesinitial}
 \frac{x}{2}\le y\le \frac{x}{2}+(1-x)^2.
\end{align}

Our first step in verifying \eqref{bryantyxestimatesinitial} for $x \in [\frac{3}{4},1]$ is to find an $\epsilon>0$ so that \eqref{bryantyxestimatesinitial} holds 
on $[1-\epsilon,1]$. To find such an $\epsilon>0$, note that $f$ must be smooth close to $x=1$ because it describes an unstable manifold of a hyperbolic critical point of a 
smooth vector field. By looking at the linearisation of \eqref{xysystmebryant} at $(1,\frac{1}{2})$, we see that the unstable manifold 
points in the direction of $(2,1)$, so $f'(1)=\frac{1}{2}$. 
Now the function $f$ also satisfies 
\begin{align}\label{unstablemanifoldequation}
(-x+f(x)+f(x)x^2)f'(x)=-xf(x)^2+2x^2f(x)^3
\end{align}
for $x$ close to $1$. Using the equalities $f(1)=\frac{1}{2}$ and $f'(1)=\frac{1}{2}$, we can write $f(x)=\frac{x}{2}+a(x-1)^2+O(x-1)^3$. Then 
\eqref{unstablemanifoldequation} becomes 
\begin{align*}
\frac{x-1}{2}+(x-1)^2\left(\frac{3}{4}+3a\right) +O(x-1)^3=\frac{x-1}{2}+(x-1)^2\left(\frac{7}{4}+\frac{a}{2}\right)+O(x-1)^3,
\end{align*}
so $a=\frac{2}{5}\in (0,1)$. The existence of the required $\epsilon>0$ follows. 

Consider the quantities $M_1=y-\frac{x}{2}$ and $M_2=y-\frac{x}{2}-(x-1)^2$. 
We find that that $y'=0$ whenever $M_1=0$, so that $M_1'\ge 0$ whenever $M_1=0$ and $x\in [0,1]$. On the other hand, we have 
\begin{align*}
 M_2'=-xy^2+2x^2y^3-(\frac{1}{2}+2(x-1))(-x+y+yx^2),
\end{align*}
which is non-positive whenever $M_2=0$ and $x\in [\frac{3}{4},1]$.
The required estimates follow. 
\end{proof}
\begin{prop}\label{FinalEstff}
 We have $x-x^2\le f(x)\le x$ for $x\in [0,\frac{1}{4}]$. In fact, $f(x)\le x$ for all $x\in [0,1]$
\end{prop}
\begin{proof}
As before, it suffices to consider the unstable trajectory of the system \eqref{xysystmebryant} which travels from $(1,\frac{1}{2})$ to $(0,0)$. 
The estimate $f(x)\le x$ follows from the fact that $f(1)<1$, and the observation that whenever the quantity $M_1=y-x$ vanishes and $x\in (0,1)$, we have  
\begin{align*}
 M_1'&=xy^2(2xy-1)-M-yx^2\\
 &=2x^3\left(x^2-1\right)<0\\
\end{align*}
so that $M_1\le 0$ is preserved. On the other hand, consider the quantity
$M_2=\frac{y}{x}-1+x$, so that 
\begin{align*}
 M_2'&=\frac{y'x-x'y}{x^2}+x'\\
 &=\frac{(-xy^2+2x^2y^3)x-(-x+y+yx^2)y}{x^2}-x+y+yx^2\\
 &=-2y^2+2xy^3+(1-\frac{y}{x}-yx)\frac{y}{x}-x+y+yx^2\\
 &=-2y^2+2xy^3+(x-M_2-yx)(M_2+1-x)-x+y+yx^2.
\end{align*}
Whenever $y=x-x^2$ and $x\in (0,\frac{3}{10}]$, we find that $M_2'> 0$. It therefore suffices to show that $y>0.21$ when $x=0.3$. 
To show this point, we consider the quantity $M_3=(y-0.21)+\frac{0.3-x}{5}$. 
Then 
\begin{align*}
 M_3'=-xy^2+2x^2y^3+\frac{x-y-yx^2}{5}
\end{align*}
so that whenever $M_3=0$ and $x\in [0.3,1]$, we have $M_3'>0$. Since $M_3>0$ at the point $(1,\frac{1}{2})$, we have that $M_3>0$ when $x=0.3$, i.e., 
our solution curve has $y>0.21$ when $x=0.3$. 
\end{proof}
We conclude this appendix with some short-time estimates on the original functions solving \eqref{bryantfirstchange} with $\delta_1=1$. 
\begin{thm}\label{BryantSolitonsmalltime}
 If $\delta_1=1$, then up until time $\frac{1}{9}$, 
 we have 
 \begin{align*}
  \frac{\sin(\sqrt{6}t)}{\sqrt{6}}&\le z(t)\le t,\\
  1-2 \tan^2\left(\sqrt{\frac{3}{2}}t\right)&\le x(t)\le 1-3t^2e^{-9t^2}.
 \end{align*} 
\end{thm}
\begin{proof}
For the $z$ estimate, note that $z(t)\le t$ follows from curvature positivity of the Bryant soliton. 
For the lower bound, we use Proposition \ref{initialestaimtesxy} to conclude that $y\ge \frac{x}{2}$, so that 
\begin{align*}
 z''&=\frac{1}{yz}\left(-x+y+yx^2\right)\\
 &\ge \frac{x}{yz}\left(-1+\frac{(x^2+1)}{2}\right)\\
 &\ge \frac{2}{z}\left(-1+\frac{(x^2+1)}{2}\right)\\
 &=\frac{(z')^2-1}{z}.
\end{align*}
Now since $R(t)=\frac{1}{t}+t+O(t^3)$, we find that 
$z(0)=0=z''(0)$, and $z'(0)=1$ with $z'''(0)=-6$. We therefore claim that 
\begin{align}\label{zlower}
 z(t)\ge \frac{\sin(\sqrt{6}t)}{\sqrt{6}},
\end{align}
 provided $t\in [0,\frac{1}{9}]$. To see this, note that the solution $\frac{\sin(\sqrt{p}t)}{\sqrt{p}}$ for $\sqrt{p}t\in [0,\frac{\pi}{2}]$ corresponds 
 to a curve in $(z,z')$ space starting at $(0,1)$, with $z$ increasing and $z'$ decreasing to the point $(\frac{1}{\sqrt{p}},0)$. We write this curve as $z'=f_p(z)$ for 
 each $p>0$. 
 The equality $z'''(0)=-6$ then implies that $z'\ge f_6(z)$, 
 as long as $z\in [0,\frac{1}{\sqrt{6}}]$. The solution of $z'=f_6(z)$ is precisely $\frac{\sin(\sqrt{6}t)}{\sqrt{6}}$; estimate \eqref{zlower} 
 follows for all $t$ so that $\sqrt{6}t\in [0,\frac{\pi}{2}]$.

We now move on to the $x$ estimates. Using $z(t)\le t$ and Proposition \ref{initialestaimtesxy}, we find that 
\begin{align*}
 x'&\le \frac{1}{t}\left(\frac{(x-1)(2x^3-x^2+3x-2)}{2\left(\frac{x}{2}+(1-x)^2\right)}\right)\\
 &\le \frac{2(x-1)}{t}\left(1-3(1-x)\right),
\end{align*}
provided $x\in [\frac{3}{4},1]$. Using \eqref{Bryantsolutionnonrescaled} and the fact 
that $R(t)=\frac{1}{t}+t+O(t^3)$, we find that $L_2(t)=\frac{1}{t}-2t+O(t^3)$. One can then verify that $x(0)=1,x'(0)=0,x''(0)=-6$. 
It is clear that $(1-x)\ge x_*(t)$, where $x_*(t)'=\frac{2x_*(t)}{t}(1-3x_*(t))$ with $x_*(0)=x_*'(0)=0$, $x_*''(0)=6$. We can easily show that 
$x^*(t)\le 3t^2$ for all $t\ge 0$, so that 
$x_*(t)'\ge \frac{2x_*(t)}{t}(1-9t^2)$; solving this implies that $(1-x)(t)\ge 3t^2e^{-9t^2}$. 
For the $x$ lower bound, we estimate 
\begin{align*}
 x'&\ge \frac{-x+\frac{x}{2}(1+x^2)}{yz}\\
 &\ge \frac{-x+\frac{x}{2}(1+x^2)}{\frac{xz}{2}}\\
 &=\frac{x^2-1}{z}\\
 &\ge \frac{\sqrt{6}(x^2-1)}{\sin(\sqrt{6}t)}\\
 &\ge \frac{2\sqrt{6}(x-1)}{\sin(\sqrt{6}t)},
\end{align*}
which implies that $(1-x)(t)\le 2 \tan^2(\sqrt{\frac{3}{2}}t)$. 
\end{proof}

\end{document}